\documentclass[twoside,11pt]{article}
\usepackage[utf8]{inputenc}
\usepackage[T1]{fontenc}
\usepackage{setspace}
\usepackage{comment}
\usepackage{fancyhdr}
\usepackage{mathtools}

\usepackage[paper=a4paper,left=32mm,right=32mm,top=30mm,bottom=30mm]{geometry}
\usepackage{tikz}
\usepackage{tikz-cd}
\usepackage{amsmath,times}
\usepackage{amsmath}
\usepackage{amssymb}
\usetikzlibrary{matrix,arrows}
\usepackage{float}
\usepackage{amsthm}
\usepackage{enumitem}
\usepackage{authblk}
\usepackage{xcolor}

\usepackage{tikz, tikz-3dplot, pgfplots}
\usepackage{tikz-cd}
\usetikzlibrary{calc}
\usetikzlibrary{positioning, calc, arrows, decorations.markings, decorations.pathreplacing}

\usepackage{hyperref}

\newtheorem{theorem}{Proposition}[section]
\newtheorem{T}[theorem]{Theorem}
\newtheorem{lemma}[theorem]{Lemma}
\newtheorem{cor}[theorem]{Corollary}

\theoremstyle{definition}
\newtheorem*{T*}{Theorem}
\newtheorem*{lemma*}{Lemma}
\newtheorem*{theorem*}{Proposition}
\newtheorem{defi}[theorem]{Definition}

\newtheorem{example}[theorem]{Example}
\newtheorem*{claim}{Claim}

\newtheorem{remark}[theorem]{Remark}

\newcommand{\N}{\mathbb{N}}
\newcommand{\Z}{\mathbb{Z}}
\newcommand{\R}{\mathbb{R}}

\newcommand{\supp}{\rm supp}

\DeclareMathOperator{\im}{im}

\DeclareMathOperator{\SL}{SL}
\DeclareMathOperator{\PSL}{PSL}

\DeclareMathOperator{\Or}{O}
\DeclareMathOperator{\SU}{SU}

\DeclareMathOperator{\Po}{P}

\DeclareMathOperator{\A}{A}
\DeclareMathOperator{\B}{B}

\DeclareMathOperator{\arc}{arc}

\DeclareMathOperator{\Sym}{S}

\DeclareMathOperator{\GL}{GL}

\DeclareMathOperator{\SO}{SO}
\DeclareMathOperator{\Fix}{Fix}

\DeclareMathOperator{\rk}{rk}
\DeclareMathOperator{\Diff}{Diff}
\DeclareMathOperator{\colim}{colim}

\DeclareMathOperator{\diam}{diam}
\DeclareMathOperator{\Cone}{Cone}

\title{ON CONJUGATION INVARIANT NORMS, ASYMPTOTIC CONES, METRIC ULTRAPRODUCTS AND CONTRACTIBILITY}
\date{%
}
\author{BASTIEN KARLHOFER}

\begin{document}
\maketitle

\begin{abstract}
In the present paper we prove lemmata on strong contractibility in asymptotic cones and metric ultraproducts which we apply to both the case of finitely generated word norms and the case of conjugation invariant norms. We recover classically known contractibility results on free products and prove the contractibility of the asymptotic cone of the infinite symmetric group $\Sym_\infty$ equipped with a conjugation invariant norm. Furthermore, we give examples of contractible metric ultraproducts arising from subgroups of general linear groups. Additionally, we discuss algebraic properties of groups arising as asymptotic cones for conjugation invariant norms. For example, we show that the asymptotic cone of the infinite symmetric group is itself an algebraically simple group relating strongly to the universally sophic groups defined by Elek and Szabo.  
\end{abstract}

\tableofcontents

\section{Introduction}

The study of asymptotic cones on a given group $G$ originates from work of Gromov in \cite{Gromov} to prove that groups of polynomial growth are virtually nilpotent. Other notable early work on asymptotic cones includes \cite{Dries}. Since then the study of asymptotic cones and their properties has become a central topic in geometric group theory with asymptotic cones encoding rich information about the underlying structure of the group $G$.

In general, the asymptotic cone depends strongly on the particular choice of ultrafilter. It is shown in \cite{Shelah} that the size of the continuum is the maximal number of pairwise non-isometric asymptotic cones for a finitely generated group if the continuum hypothesis is true. If the continuum hypothesis is not true, then \cite{Shelah} shows that $\SL(n,\R)$ for $n \geq 3$ is an example of a group with $2^{2^{\aleph_0}}$ pairwise non-homeomorphic asymptotic cones. An example of a finitely presented group with two non-homeomorphic asymptotic cones that is independent of the continuum hypothesis is given in \cite[Thm. 1.1]{Sapir}.

It has first been noted in \cite[Prop. 3.3]{Zhuang} that the multiplication of representative sequences turns the asymptotic cone into a group whenever the underlying metric is bi-invariant. In the context of symplectic geometry where the group of Hamiltonian diffeomorphism carries a naturally bi-invariant norm, called the Hofer norm, asymptotic cones have been used by {\'A}lvarez-Gavela et Al. in \cite{Ham} to study embedding questions of the free group.

The growth behaviour of conjugacy classes has been a classical field of study. This growth behaviour is analysed in the theory of finite simple groups under the term covering number (see for example \cite{Liebeck}, \cite{Liebeck2}) and has been studied more recently under the terminology of (strong/uniform) bounded generation in algebraic groups \cite{KLM}.  

The construction of asymptotic cones for a sequence of spaces $X_n$ and metric ultraproducts are very closely related. However, in the latter case it is common practice to consider spaces $X_n$ of finite diameter and normalise the norms such that $\diam(X_n)=1$ for all $n$, which makes every sequence of elements admissible. Background on the study of metric ultraproducts for sequences of classical groups can be found in \cite{Stolz}, \cite{Wilson} and \cite{Wilson2}. In \cite{Wilson} it is shown that metric ultraproducts of finite simple groups are simple, thus extending the results of Elek and Szabo in \cite{Elek}. 

In Section 2 we discuss preliminaries on conjugation invariant norms, asymptotic cones on groups and recall that asymptotic cones form groups themselves if the underlying metrics are bi-invariant. For completeness, we prove that the classical study of finitely generated word norms and word norms generating by finitely many conjugacy classes has quasi-isometrically trivial intersection. Moreover, we give an example of a word norm associated to an infinite generating set on the integers and show that the asymptotic cone with respect to this word norm has non-trivial 2-torsion. 

\begin{theorem*}[\ref{Z-torsion}]
The asymptotic cone $\Cone_\omega(\Z, \|.\|_S)$ contains non-trivial elements of order two. In particular, $\Cone_\omega(Z, \|.\|_S)$ and $\R$ are not isomorphic. 
\end{theorem*} 

This example visualises that passing to an infinite set of conjugacy classes can create very wild asymptotic behaviour. 

In Section 3 we give conditions in Lemma \ref{SES-Lemma} and Lemma \ref{SES-fcg} when short exact sequence on the level of groups equipped with conjugation invariant norms give rise to short exact sequences on the level of asymptotic cones. 

\begin{lemma*}[\ref{SES-fcg}]
Let $ 1 \rightarrow H \xrightarrow{f} G \xrightarrow{g} G/H \rightarrow 1$ be a short exact sequence of groups and let $\|.\|_G$ and $\|.\|_{G/H}$ be conjugation invariant word norms generated by finitely many conjugacy classes. If $f$ is a quasi-isometric embedding, then there is the following short exact sequence on the level of asymptotic cones
\begin{equation*} 
\centering
\begin{tikzcd}
1  \arrow[-latex]{r} & \Cone_\omega(H,\|.\|_H) \arrow[-latex]{r}{\widehat{f}} 
 & \Cone_\omega(G,\|.\|_G) \arrow[-latex]{r}{\widehat{g}} & \Cone_\omega( G/H , \|.\|_{G/H}) \arrow[-latex]{r} & 1.
\end{tikzcd}
\end{equation*}
\end{lemma*}

We derive some topological product decompositions using Lemma \ref{Serre} and Lemma \ref{qmshortexseq}. In Proposition \ref{producthomeo} we prove that all topologically meaningful information is contained in the commutator subgroup $[G,G]$ of a group $G$.

\begin{theorem*}[\ref{producthomeo}]
Let $G$ be a group. Let $\|.\|$ and $\|.\|_{G^{ab}}$ be word norms generated by finitely many conjugacy classes on $G$ and its abelianisation $G^{ab}$. Then there is a homeomorphism 
\begin{align*}
\Cone_\omega(G,\|.\|) \cong \Cone_\omega([G,G],\|.\|) \times (\R^k, \ell_1),
\end{align*} 
where $k=\rk(G^{ab})$. 
\end{theorem*}

In Section 4 we prove a general contraction lemma for the asymptotic cone in Lemma \ref{GeneralContractionLemma} that we apply to direct sums, the infinite symmetric group and free products of groups. As an application to the infinite symmetric group we obtain the following theorem.

\begin{T*}[\ref{Symcontractible}]
All asymptotic cones of $(\Sym_\infty, \|.\|_{\supp})$ are contractible.
\end{T*}

In Section 5 we turn to algebraic properties of the asymptotic cones of the infinite symmetric group $\Cone_\omega(\Sym_\infty, \|.\|_{\supp})$. In Proposition \ref{symmetricgroupthm} we prove that these cones are uniformly perfect of commutator length one and in Proposition \ref{Sinftysimple} we show that they are simple groups. The related original results on universally sophic groups in \cite{Elek} only cover the case of sequences of finite symmetric groups.  Moreover, we show that the induced continuously valued norm on $\Cone_\omega(\Sym_\infty, \|.\|_{\supp})$ is coarsely equivalent to any word norm generated by finitely many conjugacy classes and can therefore be considered as a natural choice to understand the geometry of $\Cone_\omega(\Sym_\infty, \|.\|_{\supp})$.

In Section 6 we discuss examples of classical metric ultraproducts their contractibility. For example, we obtain in Proposition \ref{SOncontractible} that 

\begin{theorem*}[\ref{SOncontractible}]
The metric ultraproduct $\prod_{n \to \omega} (\SO(n),\|.\|_{\overline{\rk}})$ of the special orthogonal groups with the projectivised rank metric is contractible.
\end{theorem*}

Finally, we give some examples of non-contractible asymptotic cones in the Appendix.

\section{Preliminaries}

\subsection{Basics on conjugation invariant word norms}

\begin{defi}
Let $G$ be a group. A map $\psi \colon G \rightarrow \R$ is called a \textit{quasimorphism} if there exists a constant $D \geq 0$ such that 
\begin{align*}
|\psi(gh)- \psi(g) - \psi(h)| \leq D \hspace{1mm} \text{for all} \hspace{1mm} g,h \in G.
\end{align*}
The \textit{defect} of $\psi$ is defined to be the smallest number $D(\psi)$ with the above property.
A quasimorphism is \textit{homogeneous} if it satisfies $\psi(g^n)=n \psi(g)$ for all $g \in G$ and all $n \in \Z$.
\end{defi}

By induction we have $| \psi(g_1 \cdots g_n) - \sum_{i=1}^n \psi (g_i) | \leq (n-1)D$ for all $g_1, \dots, g_n \in G$. It follows from the definition that sums and scalar multiples of (homogeneous) quasimorphisms are (homogeneous) quasimorphisms again. It is immediate that homogeneous quasimorphisms vanish on elements of finite order and are unbounded if there exists an element $g \in G$ such that $\psi(g) \neq 0$.    

\begin{defi}
Let $\psi \colon G \to \R$ be a quasimorphism. The \textit{homogenisation} $\bar{\psi} \colon G \to \R$ of $\psi$ is defined to be $\bar{\psi}(g)= \lim_{n \in \N} \frac{\psi(g^n)}{n}$ for all $g \in G$.
\end{defi}

\begin{lemma}[{\cite[p.18]{Calegari}}] \label{HomogenisationQM}
The homogenisation $\bar{\psi}$ of a quasimorphism $\psi \colon G \to \R$ is a homogeneous quasimorphism. Moreover, it satisfies $|\bar{\psi}(g)- \psi(g)| \leq D(\psi)$ for any $g \in G$. 
\end{lemma}

Homogeneous quasimorphisms are constant on conjugacy classes by \cite[p.19]{Calegari}. A perfect group $G$ is called \textit{uniformly perfect} if there exists a constant $K$ such that every $g \in G$ can be expressed as a product of at most $K$ commutators. It follows by elementary calculation, that a uniformly perfect group $G$ does not admit an unbounded quasimorphism.

\begin{defi}
A function $\nu \colon G \to \R$ is called a \textit{norm} on $G$ if $\nu$ satisfies for all $g, h \in G$ that
\begin{itemize}
\item $\nu(g) \geq 0$, 
\item $\nu(g)=0$ if and only if $g=1$, 
\item $\nu(gh) \leq \nu(g)+ \nu(h)$,
\end{itemize}
If in addition $\nu$ satisfies for all $g,h \in G$ 
\begin{itemize}
\item $\nu(hgh^{-1}))= \nu(g)$, 
\end{itemize}
then it is called \textit{conjugation invariant}. The supremum $\nu(G)= \sup \{\nu(g) \mid g \in G \}$ is called the \textit{diameter} of the norm $\nu$. If $\nu(G)= \infty$, then $\nu$ is called \textit{unbounded}.  If there exists $g \in G$ such that $\lim_{n \to \infty} \frac{\nu(g^n)}{n} > 0$, then $\nu$ is called \textit{stably unbounded}.
\end{defi}

\begin{example}
Let $G$ be a group together with a generating set $S$. The \textit{word norm} generated by $S$ is the norm on $G$ defined by 
\begin{align*}
\nu_S(g)= \min \{ n \mid g= s_1 \cdots s_n \hspace{1mm} \text{where} \hspace{1mm} n \in \N \hspace{1mm} \text{and} \hspace{1mm} s_i \in S \hspace{1mm} \text{for all} \hspace{1mm} i \}.
\end{align*}
If we assume additionally that the set $S$ is invariant under conjugation, then $\nu_S$ is a conjugation invariant norm. 
\end{example}

Note that, if $H$ is a subgroup of $G$ and a normal generating set $S$ of $H$ normally generates $G$ as well, the norms generated by $S$ on $H$ and on $G$ can differ significantly from one another. One example is $SL(2,\mathbb{Z})$ inside $SL(n,\mathbb{Z})$ for $n \geq 3$. Then $SL(2, \mathbb{Z})$ generates all higher $SL(n,\mathbb{Z})$ under conjugation, but $SL(n, \mathbb{Z})$ is bounded for $n \geq 3$ (see for example \cite{Morris}, originally due to \cite{Carter}). However, $\SL(2,\Z)$ admits unbounded conjugation invariant norms since it possesses many unbounded quasimorphisms. In fact, $\SL(2,\Z)$ even admits unbounded quasimorphisms that are invariant under all automorphisms by \cite[Example 6.6]{ich1}. 

\begin{defi}
Two conjugation invariant norms $\|.\|_1$ and $\|.\|_2$ are called \textit{equivalent} if there exist Lipschitz-constants $A,B >0$ such that $\|.\|_1 \leq A \|.\|_2$ and $\|.\|_2 \leq B\|.\|_1$. Furthermore, $\|.\|_1$ and $\|.\|_2$ are \textit{coarsely equivalent} if there exist constants $A > 0$ and $B \geq 0$ such that 
\begin{align*}
\frac{1}{A} \|.\|_1 - B \leq \|.\|_2 \leq A\|.\|_1 + B
\end{align*}
\end{defi}

\begin{lemma}\label{fcgmaximal}
Let $S \subset G$ be a generating set and $\nu_S$ be the conjugation invariant word norm defined by conjugacy classes of $S$. Then any other conjugation invariant norm $\mu \colon G \rightarrow \R$ for which $C_\mu = \sup_{s \in S} \|s\| < \infty$ satisfies $\mu \leq C_\mu \nu_S$. In particular, this always holds if $S$ is finite. 
\end{lemma}
 
\begin{proof}
Given $g \in G$ with $\nu_S(g)=n$ we write $g=s_1 \cdots s_n$ for $s_i$ conjugates of elements in $S$ and calculate
\begin{align*}
\mu(g)=\mu(s_1 \cdots s_n) \leq \sum_{i=1}^n \mu(s_i) \leq nC_\mu = C_\mu \nu(g).
\end{align*}
\end{proof} 

\begin{defi}
A group $G$ is called \textit{bounded} if $G$ does not admit a conjugation invariant norm that has infinite diameter.
\end{defi}

By Lemma \ref{fcgmaximal} any finitely normally generated group $G$ is bounded if there exists one word norm generated by finitely many conjugacy classes that has finite diameter on $G$.
Another immediate consequence of Lemma \ref{fcgmaximal} is the following lemma.

\begin{lemma}\label{equiva}
On a group $G$ any two generating sets $S$ and $R$ consisting of finitely many conjugacy classes define equivalent conjugation invariant word norms $\nu_S$ and $\nu_R$. \qed
\end{lemma}

This means that although generating sets consisting of finitely many conjugacy classes usually have infinite cardinality, there is an analogue of the classic statement that any two finite generating sets on a group $G$ will yield quasi-isometric Cayley graphs. However, if one relaxes this assumption further to allow generating sets $S$ consisting of infinitely many conjugacy classes, many problems arise even if the group remains unbounded with respect to $S$. A particularly pathological example is given in Proposition \ref{Z-torsion}.

\begin{lemma} \label{Autqm implies Autnorm}
Let $\psi \colon G \to \R$ be a quasimorphism with unbounded image, but bounded on a generating set $S$ of $G$. Then the word norm $\|.\|_S$ is stably unbounded on $G$.
\end{lemma}

\begin{proof}
By Lemma \ref{HomogenisationQM} we can assume that $\psi$ is homogeneous. Let $K$ be a positive bound for the absolute value of $\psi$ on $S$. Write $g \in G$ as a product of these generators $g=s_1 \cdots s_n$ for some $n \in \N$ where $s_i \in S$. Then 
\begin{align*}
| \psi(g)|  = | \psi(s_1 \cdots s_n) | \leq |\psi(s_1)|+ \dots + |\psi(s_n)| \hspace{-0.1mm}+(n-1)D(\psi) 
\leq n(K + D(\psi)) 
\end{align*}
shows that $\|g\|_{S} \geq \frac{|\psi(g)|}{K+D(\psi)}$ for all $g \in G$. Since $\psi$ is homogeneous, $\|g^k\|_{S} \geq k \cdot \frac{|\psi(g)|}{K+D(\psi)}$  for all $k \in \N$, $g \in G$. Since $\psi$ does not vanish everywhere, $\|.\|_{S}$ is stably unbounded on $G$. 
\end{proof} 

\subsection{Conjugation invariant norms on the infinite symmetric group}

Define the infinite symmetric group $\Sym_\infty$ to be the colimit over the inclusions of the symmetric groups $\Sym_n \hookrightarrow \Sym_{n+1}$. That is, $\Sym_\infty$ is the group of finitely supported permutations of the natural numbers. Analogously, define the infinite alternating group $\A_\infty$ as the group of finitely supported alternating permutations. Since $\A_n$ is simple for any $n$ it follows that the colimit $\A_\infty$, which has index 2 in $\Sym_\infty$, is as well. Recall, that every permutation has a unique \textit{cycle decomposition}. We follow the classical convention that in cycle notation cycles are performed from \textit{left to right}, that is $(x \enspace y)(y \enspace z) = (x \enspace z \enspace y)$.

\begin{defi}
Define the support norm $\|.\|_{\rm supp}$ on $\Sym_\infty$ and $\A_\infty$ by
\begin{align*}
 \|\sigma\|_{\rm supp} = 
 \begin{cases}
 | \lbrace n \in \mathbb{N}: \enspace \sigma(n) \neq n \rbrace | &  \text{if} \hspace*{2mm} \sigma \neq id, \\
 0 & \text{if} \hspace*{2mm} \sigma=id.
 \end{cases}
\end{align*}
\end{defi}

It is elementary to verify that this defines a conjugation invariant norm on both $\Sym_\infty$ and $\A_\infty$. Moreover, the inclusion $\A_\infty \hookrightarrow \Sym_\infty$ is an isometric embedding and quasi-surjective with constant 1 with respect to the support norms. Hence, it will follow in Proposition \ref{qi-invariant} that $\Cone_\omega(\Sym_\infty,\|.\|_{\rm supp})=\Cone_\omega(A_\infty, \|.\|_{\rm supp})$.

\begin{remark}
If one considers $\Sym_\infty$ as a subgroup of the group $\GL_\infty(k)$ arising as the colimit over the finite general linear groups for any field $k$, then the restriction of the rank norm defined by $\|g\|_{\rk} = \rk(g -id)$ for any $g \in GL_\infty$ is equivalent to the support norm on $\Sym_\infty$. In fact, it is clear that for any $\sigma \in \Sym_\infty$ we have $\|\sigma \|_{\rk} \leq \| \sigma \|_{\rm supp}$ and Lemma \ref{1} can be used to obtain $\|.\|_{\rk} \geq \frac{1}{3} \|.\|_{\supp}$. This implies the equivalence of the norms $\|.\|_{\rk}$ and $\|.\|_{\supp}$ on $\Sym_\infty$. 
\end{remark}

\begin{defi}
Define the transposition norm $\|.\|_{tr} \colon \Sym_\infty \to \R_{\geq 0}$ to be the word norm generated by the set $T$ consisting of all transpositions. Since $T$ consist of a single conjugacy class in $\Sym_\infty$, the norm $\|.\|_{tr}$ is conjugation invariant.
\end{defi}

\begin{lemma}\label{symequiva}
The two norms $\|.\|_{tr}$ and $\|.\|_{\rm supp}$ on $S_\infty$ satisfy 
\begin{align*}
\|.\|_{tr} \leq \|.\|_{\rm supp} \leq 2\|.\|_{tr}.
\end{align*}
Consequently, $\|.\|_{\rm supp}$ is equivalent to any word norm $\nu$ on $\Sym_\infty$ that is generated by finitely many conjugacy classes.
\end{lemma}

\begin{proof}
Since the support of every transposition equals 2, the second inequality is immediate. We now prove $\|.\|_{tr} \leq \|.\|_{\rm supp}$ inductively on the size of the support. Let $\sigma \in S_\infty$ with $\|\sigma \|_{\rm supp}=n$ be given. If $n \leq 1$ then $\sigma$ is in fact the identity and the inequality holds trivially. Now assume $n > 1$. Let $x \in \N$ be such that $\sigma(x)=y \neq x$. and let $z= \sigma^{-1}(x)$. Then $z \neq x$. Recall, that we use the convention of multiplying permutations from left to right, which mean first the left one, then the right one. If $z=y$, write $\sigma = (x \enspace y) \sigma^\prime$ for some $\sigma^\prime$ with $\|\sigma^\prime\|_{\rm supp} \leq n-2$ and apply the induction hypothesis. If $z \neq y$, consider $(y \enspace z) \sigma $ which has the property that it maps $y$ to $x$ and $x$ to $y$ whence $(y \enspace z) \sigma = (x \enspace y) \sigma^\prime $ for some $\sigma^\prime$ with $\|\sigma^\prime\|_{\rm supp} \leq n-2$ which shows by induction that $\|\sigma\|_{tr} = \|(y \enspace z) (x \enspace y) \sigma^\prime \|_{tr} \leq 2 + n-2 = n$. Therefore, $\|.\|_{\rm supp}$ and $\|.\|_{tr}$ are equivalent and the result follows from Lemma \ref{equiva}.
\end{proof}

\begin{defi}
Any two 3-cycles are conjugate in $\A_\infty$ and the set of 3-cycles normally generates $\A_\infty$. Define $\|.\|_{3}$ to denote the conjugation invariant word norm on $\A_\infty$ generated by the single conjugacy class of 3-cycles.
\end{defi}

The following lemma verifies that under the restriction of conjugation invariant norms to the subgroup $\A_\infty$ of $\Sym_\infty$ we still stay within the equivalence class of word norms generated by finitely many conjugacy classes. Hence, by \cite[Lemma 10.83]{KapDru}, which will be recalled in Proposition \ref{qi-invariant}, it will follow that all asymptotic cones of $\A_\infty$ and $\Sym_\infty$ are Lipschitz equivalent as metric spaces and isomorphic as groups.

\begin{lemma}\label{alternatingequiva}
The two norms $\|.\|_3$ and $\|.\|_{tr}$ on $A_\infty$ are equivalent satisfying $\|.\|_{tr} \leq 2\|.\|_{3}$ and $\|.\|_{3} \leq \frac{3}{2}\|.\|_{tr}$. Moreover, $\|.\|_{\rm supp}$ is equivalent to any word norm $\nu$ that is generated by finitely many conjugacy classes on $\A_\infty$.
\end{lemma}

\begin{proof}
Given $\sigma \in \A_\infty$ with $\|\sigma\|_{3}=n$. Write $\sigma=c_1...c_n$ as a product of 3-cycles $c_i$ and calculate: 
\[
\|\sigma\|_{tr} = \|c_1...c_n\|_{tr} \leq \sum_{i=1}^n \|c_i\|_{tr} \leq 2n=2\|\sigma\|_{3} \hspace{1mm},
\]
where we just used that any 3-cycle is the product of two transpositions. 

For the other inequality we first recall that 
the product of any two transpositions $\tau_1$, $\tau_2$ satisfies $\|\tau_1 \tau_2\|_{3} \leq 3$. Namely, since every transposition has support norm equal to two, we have that if the support of the two transpositions is equal, then their product is the identity. If their supports do not agree and are not disjoint either then the product yields a 3-cycle. If the supports are disjoint then we can write $\tau_1=(x_1 \enspace y_1)$ and $\tau_2=(x_2 \enspace y_2)$ and pick $z \in \mathbb{N}$ different from all of them. We have that
\begin{align*}
\tau_1 \tau_2 = (x_1 \enspace y_1) (x_2 \enspace y_2) = (z \enspace y_1 \enspace x_1)(x_2 \enspace z \enspace y_1 )(y_2 \enspace x_2 \enspace z)
\end{align*}
 is a product of three 3-cycles. So now suppose that
we are given $\sigma \in \A_\infty$ with $\|\sigma\|_{tr}=2n$. Then $\sigma=\tau_1...\tau_{2n}$ and the following calculation shows the other inequality
\[
\|\sigma\|_{3} =\|\tau_1...\tau_{2n}\|_{3} \leq \sum_{i=1}^{n}\|\tau_{2i-1} \tau_{2i} \|_{3} \leq 3n= \frac{3}{2}\|\sigma\|_{tr} \hspace{1mm}. 
\]
The second statement follows from the equivalence of $\|.\|_{tr}$ and $\|.\|_{\rm supp}$ on $\Sym_\infty$ together with the fact that all word norms on $\A_\infty$ generated by finitely many conjugacy classes are equivalent by Lemma \ref{equiva}.
\end{proof}

\subsection{Ultrafilters}

\begin{defi}
Let $I$ be a set. A \textit{filter} $\omega$ on $I$ is a set $\omega$ consisting of subsets of $I$ such that
\begin{enumerate}[label=(\roman*)]
\item $\emptyset \notin \omega$ ,
\item $A \subset B \subset I$ and $A \in \mathfrak{B}$ implies $B \in \omega$ , 
\item $A,B \in \omega$ implies $A \cap B \in \omega$ . 
\end{enumerate}

The filter $\omega$ is called an \textit{ultrafilter} if it additionally satisfies

\begin{enumerate}[label=(\roman*), resume]
\item$\forall A \subset I$ either $A \in \omega$ or $(I \setminus A) \in \omega$ .
\end{enumerate}
\end{defi}

\begin{defi}
Let $\omega$ be a non-principal ultrafilter on the set $I$ and let $Y$ be a topological space. Let $f \colon I \rightarrow Y$ be a function. Then a point $y \in Y$ is called ultralimit of the function $f$, denoted by $y= \lim_\omega f(i)$ if it satisfies that for all $\epsilon > 0$ we have 
\begin{align*}
\lbrace i \in I \hspace{2mm} : \hspace{2mm} d(f(i),y) < \epsilon \rbrace \in \omega.
\end{align*}
From now onwards we will always take $I= \N$, identify the function $f \colon \N \to Y$ with its image sequence $(x_n)_n$ and just write $y=\lim_\omega x_n $ for the ultralimit. 
\end{defi}

\begin{remark}
Given $p \in I$ one can define an ultrafilter by setting $A \in \omega$ iff $p \in A$. Filters arising in this way are called principal. For such a princial ultrafilter the limit of every function $f \colon I \to Y$ is the image $f(p)$ (compare \cite[p.339]{KapDru}). 
\end{remark}

We will require all ultrafilters to be \textit{non-principal} from now onwards. The following properties of ultralimits are well-known (see e.g. \cite[p.338]{KapDru}). 

\begin{theorem}
Let $\omega$ be a non-principal ultrafilter on $\N$ and let $Y$ be a compact metric space. Then: 
\begin{enumerate}[label=(\roman*)]
\item  For every function $f \colon \N \rightarrow Y$ the ultralimit exists and is unique , 
\item For any convergent function $f \colon \N \rightarrow Y$  the ultralimit agrees with the usual limit, 
\item The ultralimit of any function $f \colon \N \rightarrow Y$ is an accumulation point of $f$,
\item If $g \colon Y \rightarrow Z$ is a continuous function to a metric space $Z$ then for every $f \colon \N \rightarrow Y$ it holds that
$ \lim_\omega g(f(i))= g \big ( \lim_\omega f(i) \big ) $.
\end{enumerate}
\end{theorem} 

Moreover, ultralimits satisfy the usual rules of calculus \cite[p.338]{KapDru}. For example, let $(a_n)_n , (b_n)_n \subset \R$ be two sequences such that $\lim_\omega a_n$ and $\lim_\omega b_n$ are both real numbers. Then 
\begin{align*}
\lim_\omega a_n + b_n & = \big ( \lim_\omega a_n \big ) + \big( \lim_\omega b_n \big),
& \lim_\omega a_n b_n  = \big ( \lim_\omega a_n \big ) \cdot \big( \lim_\omega b_n \big).
\end{align*} 
Also, if two sequences $(x_n)_n, (y_n)_n \subset Y$ agree for a subset of indices $A \in \omega$, then the ultralimits $\lim_\omega x_n = \lim_\omega y_n$ agree. As an example of how to work with ultralimits, let us very they respect closed inequalities on $\omega$-heavy subsets.

\begin{lemma}\label{ultralimitinequality}
Let $(a_n), (b_n) \subset \R$ be two sequences of real numbers. Assume there is a set $A \in \omega$ such that for all $n \in A$ we have $a_n \leq b_n$. Then it holds that $\lim_\omega a_n \leq \lim_\omega b_n $.
\end{lemma} 

\begin{proof}
Let $a= \lim_\omega a_n$ and $b= \lim_\omega b_n$. Assume that $a > b$ and both are real numbers. Then set $\epsilon= \min \{ \frac{a-b}{2},1 \}$. By the definition of ultralimits the sets 
\begin{align*}
A_\epsilon = \{ n \hspace{1mm} : \hspace{1mm} |a_n-a| < \epsilon \} \hspace{2mm}, \hspace{2mm} B_\epsilon = \{ n \hspace{1mm} : \hspace{1mm} |b_n-b| < \epsilon \} 
\end{align*}
both lie in $\omega$. We have $a- \epsilon \geq b+\epsilon$ and since ultrafilters are closed under taking supersets we conclude 
\begin{align*}
\{ n \hspace{1mm} : \hspace{1mm} |a_n-a| < \epsilon \} \subset \{ n \hspace{1mm} : \hspace{1mm} a_n \geq b+ \epsilon \} \in \omega . 
\end{align*}
Intersection of the latter set with $A$ yields 
\begin{align*}
B=\{ n \hspace{1mm} : \hspace{1mm} b_n \geq a_n  \geq b+ \epsilon \} \in \omega.
\end{align*}
However, this implies $ \emptyset= B \cap B_\epsilon \in \omega$, which is a contradiction. Hence, $a \leq b$.

If $a = \infty$ then for all $C>0$ we have $A_C=\{n \hspace{1mm} : \hspace{1mm} a_n \geq C \} \in \omega$. Hence $A \cap A_C \in \omega$ which yields $ B_C=\{n \hspace{1mm} : \hspace{1mm} b_n \geq C \} \in \omega$. Thus $b=a= \infty$ and the statement follows. 
\end{proof}

\subsection{Asymptotic Cones}

Let us recall the construction of an \textit{asymptotic cone} of a sequence of metric spaces $(X_n,d_n)$. The closely related construction of \textit{metric ultraproducts} is stated in the beginning of Section 6 to keep notation consistent for now. 

Pick an ultrafilter $\omega$ and a  scaling sequence $(s_n) \subset \R_{>0}$ such that $\lim_\omega s_n = \infty$. Pick a sequence of basepoints $p_n$. Let $X_n^b(\omega, s_n, p_n)$ be the set of sequences $(x_n) \subset \in X_n$ such that the sequence $ \left( \frac{d_n(x_n,p_n)}{s_n} \right)$ is bounded. Any sequence that belongs to $X_n^b(\omega, s_n, p_n)$ is called \textit{admissible}. Then an \textit{asymptotic cone} is defined to be the metric space
\begin{align*}
\Cone_\omega(X_n,s_n,p_n) = X_n^b(\omega, s_n, p_n)/ \sim 
\end{align*}
with $(x_n) \sim (y_n)$ if $\lim_\omega \frac{d_n(x_n,y_n)}{s_n}=0$. The equivalence class of an admissible sequence $(x_n)$ is denoted by $[x_n]$. An asymptotic cone carries the induced metric arising as the ultralimit $d([x_n],[y_n]) = \lim_\omega \frac{d_n(x_n,y_n)}{s_n}$.

The asymptotic cone depends on the particular choice of ultrafilter independently of the continuum hypothesis \cite[Thm 1.1]{Osin}. For a single metric space $(X_n,d_n)=(X,d)$ the space $\Cone_\omega(X,s_n,p_n)$ is called an asymptotic cone of $(X,d)$. Our canonical choice of scaling sequence will be $s_n=n$ and in this case we will omit the scaling sequence from the notation. Moreover, if $p_n=p \in X$ is a fixed basepoint then the asymptotic cone is independent of the particular choice of fixed basepoint.

Iterating the construction of an asymptotic cone just amounts to a change of ultrafilter \cite[Cor. 10.79]{KapDru}. In the case of group equipped with norms the homogeneity of the spaces makes the construction of an asymptotic cone independent from the choice of basepoint by the following Lemma \ref{ConeL1} (compare \cite[Exercise 10.66]{KapDru}). Thus, basepoints will be suppressed in the notation after this lemma. Lemma \ref{ConeL1} and \ref{ConeL2} should be well-known, but are recalled here for completeness.

\begin{lemma} \label{ConeL1}
Let $G_n$ be a sequence of groups together with norms $\|.\|_n$. Let $s_n$ be a scaling sequence and $p_n$ be a sequence of elements in $G$. Then 
\begin{align*}
\Cone_\omega(G_n,s_n,p_n) \cong \Cone_\omega(G_n,s_n,1)
\end{align*}
are isometric. 
\end{lemma}

\begin{proof}
Indeed, the right invariance of a metric induced by a norm means that we can just use the group multiplication to translate one sequence into another. 

Spelled out in more detail this means that we define maps 
\begin{align*}
\varphi & \colon \Cone_\omega(G_n,s_n,1)  \rightarrow \Cone_\omega(G_n,s_n,p_n), & [x_n] \rightarrow [x_n p_n], \\
\psi & \colon \Cone_\omega(G_n,s_n,p_n)  \rightarrow \Cone_\omega(G_n,s_n,1) , & [y_n] \rightarrow [y_n p_n^{-1} ].  
\end{align*}
Then, $d_n(\varphi(x_n),p_n)= \|x_n p_n p_n^{-1} \|_n=\|x_n\|_n=d_n(x_n,1)$ and $d_n(\psi(y_n),1)= \|y_n p_n^{-1} \|=d_n(y_n,p_n)$, which shows that both maps are well defined on the level of sequences. Similarly, it is immediate from 
\[
d_n(\varphi(x_n),\varphi(z_n))=\|x_n p_n (z_n p_n)^{-1} \|_n = \|x_n z_n^{-1} \|_n=d_n(x_n,z_n)
\]
that $\varphi$ is an isometry and analogously for $\psi$. Consequently, both maps are well defined and continuous on the respective asymptotic cones. By definition $\varphi$ and $\psi$ are inverse to one another. 
\end{proof} 

\begin{lemma} \label{ConeL2}
Let $(X_n, d_n, p_n)$ be a sequence of pointed metric spaces. Let $(s_n), (r_n) \subset \R_{>0}$ be two scaling sequences with $\lim_\omega \frac{s_n}{r_n}=K$ with $K \in \R_{>0} $. Then $\Cone_\omega(X_n,s_n, p_n)$ and $\Cone_\omega(X_n, r_n, p_n)$ are bi-Lipschitz homeomorphic with Lipschitz constants $K$ and $\frac{1}{K}$. In particular, for $K=1$ the cones $\Cone_\omega(X_n,s_n,p_n)$ and $\Cone_\omega(X_n,r_n,p_n)$ are isometric. 
\end{lemma}

\begin{proof}
First, let $A= \{ n | \frac{s_n}{r_n} \in [\frac{K}{2},2K] \}$. Since $\lim_\omega \frac{s_n}{r_n}=K$, it holds by definition of the ultralimit that $A \in \omega$.
Let $[g_n] \in \Cone_\omega(X_n,s_n,p_n)$ and $C_g$ be such that $d_n(g_n,p_n) \leq C_g s_n$ for all $n \in \N$. Define 
\begin{align*}
g_n^\prime = 
\begin{cases}
g_n , n \in A ,\\
p_n , n \notin A.
\end{cases}
\end{align*}
Then for $n \notin A$ we have $d_n(g_n^\prime,p_n)=0$ and for $n \in A$ we have 
\begin{align*}
d_n(g_n^\prime, p_n)= d_n(g_n,p_n) \leq C_g s_n = C_g \frac{s_n}{r_n}r_n \leq 2K C_g r_n.
\end{align*}
So, for all $n \in \N$ it holds that $d_n(g_n^\prime,p_n) \leq 2KC_g r_n$, which means that $[g_n^\prime]$ yields an element in $\Cone_\omega(X_n,r_n,p_n)$. Define 
\begin{align*}
\varphi \colon \Cone_\omega(X_n,s_n,p_n) \rightarrow \Cone_\omega(X_n,r_n,p_n) \hspace{5mm} \text{by} \hspace{5mm} \varphi([g_n])=[g_n^\prime].
\end{align*}

\begin{claim}
The map $\varphi$ is well-defined an Lipschitz continuous with constant $K$. 
\end{claim}

Let $[g_n],[h_n] \in \Cone_\omega(X_n,s_n,p_n)$. Then, by Lemma \ref{ultralimitinequality} and the product rule of ultralimits it holds that
\begin{align*}
d(\varphi([g_n]),\varphi([h_n])) & = \lim_\omega \frac{d_n(g_n^\prime,h_n^\prime)}{r_n}  \leq \lim_\omega \frac{d_n(g_n,h_n)}{r_n} = \lim_\omega \left( \frac{d_n(g_n,h_n)}{s_n} \cdot \frac{s_n}{r_n} \right) \\
& = \left( \lim_\omega \frac{d_n(g_n,h_n)}{s_n} \right) \cdot \left( \lim_\omega \frac{s_n}{r_n} \right) = K d([g_n],[h_n]).
\end{align*}
In particular, $\varphi$ does not depend on the particular choice of representative sequence. 

Conversely, define for $[g_n] \in \Cone_\omega(X_n, r_n , p_n)$ the sequence 
\begin{align*}
\widehat{g_n} = 
\begin{cases}
g_n , n \in A ,\\
p_n , n \notin A.
\end{cases}
\end{align*}
Let $D_g$ be such that $d_n(g_n,p_n) \leq D_g r_n$ for all $n$. Then, for all $n \in A$ it holds that 
\begin{align*}
d_n(\widehat{g_n}, p_n)= d_n(g_n,p_n) \leq D_g r_n = D_g \frac{r_n}{s_n}s_n \leq \frac{2}{K} D_g s_n.
\end{align*}
Since $\widehat{g_n}=p_n$ for all $n \notin A$ the above equality is in fact true for all $n \in \N$. 
We set 
\begin{align*}
\psi \colon \Cone_\omega(X_n,r_n,p_n) \rightarrow \Cone_\omega(X_n,s_n,p_n).
\end{align*}
As above, $\psi$ is well-defined and Lipschitz continuous with constant $\frac{1}{K}$.
 
If $[g_n] \in \Cone_\omega(X_n,s_n,p_n)$, then $(\psi \circ \varphi )([g_n])= [\widehat{g_n^\prime}]$. For all $n \in A$ we have $\widehat{g_n^\prime}=g_n$ and so $\frac{d_n(\widehat{g_n^\prime},g_n)}{s_n}=0$. Since $A \in \omega$, it follows that $d((\psi \circ \varphi )([g_n]),[g_n])=0$ and so $\psi \circ \varphi =id$. The equality $\varphi \circ \psi=id$ follows similarly.
\end{proof}

The previous lemma shows that the asymptotic cone only depends on the growth rate of a scaling sequence and not the sequence itself. Thus, we may assume all scaling sequences to consist of natural numbers from now onwards.

\begin{example}
The asymptotic cone of $(\Z, |.|)$ is $(\R, |.|)$. 
\end{example}

\begin{defi}
Let $(X,d)$ and $(X^\prime, d^\prime)$ be metric spaces. A map $f \colon X \rightarrow X^\prime$ is called a quasi-isometric embedding if there exist constants $A \geq 1$ and $B \geq 0$ such that for all $x,y \in X$ we have 
\begin{align*}
\frac{1}{A}d(x,y)-B \leq d^\prime( f(x), f(y)) \leq A d(x,y)+B.
\end{align*}
We call $f$ a quasi-isometry if it additionally is quasi-surjective. That is, there exists a third constant $C \geq 0$ such that for all $z \in X$ there exists $z^\prime \in X^\prime$ with $d^\prime(f(z), z^\prime) \leq C$. 
\end{defi}

\begin{theorem}[{\cite[Lemma 10.83]{KapDru}}]\label{qi-invariant}
Let $f \colon (X,d) \rightarrow (X',d')$ be a quasi-isometric embedding of metric spaces. Then this gives rise to a bi-Lipschitz embedding on the level of asymptotic cones $\Cone_\omega(X',d',s_n) \rightarrow \Cone_\omega(X,d,s_n)$. If $f$ is additionally quasi-surjective then $f$ is surjective and so $f$ is a bi-Lipschitz homeomorphism.  
\end{theorem}

In fact, the asymptotic cone of any metric space is complete \cite[Prop. 4.2]{Dries}. Furthermore, the asymptotic space of a geodesic space is geodesic itself and the asymptotic cone of a product of metric spaces $X \times Y$ is the product of the asymptotic cones $X$ and the asymptotic cone of $Y$ \cite[Cor. 1068]{KapDru}.

\begin{example}
Let $n \in \N$ and let $\|.\|$ be any word norm on $\Z^n$ generated by a finite set $S \subset \Z^n$. Then $\|.\|_S$ is Lipschitz-equivalent to the $\ell_1$-metric on $\Z^n$ and thus $\Cone_\omega(\Z^n, \|.\|_S)$ is Lipschitz-equivalent to $(\R^n, \ell_1)$.  
\end{example}

\begin{example}
Consider $G= \bigoplus_{i \in \mathbb{N}} \mathbb{Z}/i$ with the norm $\|.\|$ generated by $e_i$ the $i$-th unit vectors. Then every element in $G$ has finite order since every element is torsion. However, the element $x$ in $\Cone_\omega(G,\|.\|)$ represented by the sequence $(x_n)_{n \in \mathbb{N}}= (ne_{n^2})_{n \in \mathbb{N}}$ satisfies $\|x^k\|=k$. In fact, there are only finitely many $n \in \mathbb{N}$ satisfying  $k> \frac{n}{2}$ and for all other $n$ we have $\|x_n^k\|=k\|x_n\|$. Therefore, the ultralimit doesn't change and we obtain $\|x^k\|=k\|x\|$ implying that $x$ generates an isometric embedding $(\mathbb{Z}, |.|) \hookrightarrow \Cone_\omega(G,\|.\|)$. Furthermore, it will be shown in Proposition \ref{circleproduct} that all asymptotic cones of $G$ are a topological product of $S^1$ with another space and thus not contractible.
\end{example}

\subsection{Asymptotic cones for bi-invariant metrics}

Our primary interest lies in the case where $X_n=G$ is a group and the bi-invariant metric $d_n=d$ is induced by a conjugation invariant norm $\|.\|$ on $G$. In this case the sequence of identity elements is a canonical choice of basepoint (compare Lemma \ref{ConeL1}). Then multiplication of representative sequences gives rise to a well-defined multiplication on $\Cone_\omega(G,||.||)$, which turns the latter into a group \cite[Prop. 3.3]{Zhuang}. For completeness we will verify the continuity of the multiplication and inversion maps here to conclude that $\Cone_\omega(G,||.||)$ is a topological group. Recall, that a metric on a group $G$ is called bi-invariant if it is both right and left invariant under multiplication in the group.

\begin{lemma}
Let $(G_n,d_n)_{n \in \N}$ be a sequence of groups $G_n$ together with bi-invariant metrics $d_n \colon G_n \times G_n \rightarrow \R$. Then $\Cone_\omega(G_n,d_n)$ is a topological group for which the induced metric is bi-invariant. 

\end{lemma}

\begin{proof}
Let us verify that the inversion and multiplication maps are continuous since the rest follows as in \cite[Prop. 3.3]{Zhuang}. 
In fact, the inversion map $\Cone_\omega(G_n,d_n) \rightarrow \Cone_\omega(G_n,d_n)$ is an isometry with respect to the induced bi-invariant metric on $\Cone_\omega(G_n,d_n)$, since for all $x,y \in \Cone_\omega(G_n,d_n)$
\begin{align*}
d(x^{-1},y^{-1})=d(x^{-1}y,1)=d(y,x)=d(x,y).  
\end{align*}
Isometries are continuous maps. Furthermore, given $\epsilon > 0$ the multiplication map $m$ satisfies $m \big( B_\frac{\epsilon}{2}(x) \times B_\frac{\epsilon}{2}(y) \big) \subset B_\epsilon(xy)$ for all $x,y \in \Cone_\omega(G_n,d_n)$. Indeed, for $x^\prime \in  B_\frac{\epsilon}{2}(x)$ and $y^\prime \in B_\frac{\epsilon}{2}(y)$ we calculate 
\begin{align*}
d(x^\prime y^\prime, xy) & =d(x^{-1} x^\prime y^\prime, y) = d(x^{-1} x^\prime, y (y^\prime)^{-1}) \leq d(x^{-1} x^\prime , 1) + d(1, y(y^\prime)^{-1}) \\
& = d(x,x^\prime) + d(y,y^\prime) < \epsilon. 
\end{align*}
This implies that preimages of open sets are open which concludes the proof.
\end{proof}

\begin{remark}
If $G_n$ carries a conjugation invariant norm $\|.\|_n$ which induces the metric $d_n$ for all $n$, then the bi-invariant metric on the asymptotic cone is induced by the conjugation invariant norm $\|.\| \colon \Cone_\omega(G_n,d_n) \to \R_{\geq 0}$ defined for $x=[(x_n)_{n \in \N}] \in \Cone_\omega(G_n,d_n)$ by
\begin{align*}
 \|x  \| = \lim_\omega \frac{\|x_n \|}{n},
\end{align*}
where $\omega$ is the non-principal ultrafilter on $\N$ used to define $\Cone_\omega(G_n,d_n)$.
\end{remark}

A few immediate observations are as follows. If $G$ is abelian, then so are all asymptotic cones of $G$ for any conjugation invariant norm on $G$. If $H \leq G$ is a subgroup and $\|.\|_G$ a conjugation invariant norm on $G$, then $\Cone_\omega(H,\|.\|_G)$ is a subgroup of $\Cone_\omega(G, \|.\|_G)$. If $f \colon H \rightarrow G$ is a group homomorphism which is a quasi-isometry, then $f$ induces a bi-Lipschitz homeomorphism which is an isomorphism of groups on the level of asymptotic cones.

\begin{theorem}\label{coarseequiv}
Let $G$ be a group together with coarsely equivalent conjugation invariant norms $\|.\|_1$ and $\|.\|_2$. Then the respective cones $\Cone_\omega(G,\|.\|_1)$ and $\Cone_\omega(G,\|.\|_2)$ are bi-Lipschitz homeomorphic. 
\end{theorem}

\begin{proof}
This statement immediately follows from the definition of coarse equivalence and Proposition \ref{qi-invariant}. 
\end{proof}

\begin{lemma}\label{2}
Let $G$ be a group with conjugation invariant word norm $\|.\|$ and let $(s_n)_n$ be a scaling sequence. Then for any non-trivial $g \in \Cone_\omega(G,s_n)$ , there exists a constant $K > 0$ and an admissible sequence $(g_n)_n \subset G$ such that $\|g_n\| \geq Ks_n$ for all $n$ and $g= [(g_n)_n]$.
\end{lemma}

\begin{proof}
Let $g$ be a non-trivial element represented by a sequence $(a_n)_n$ and let $\epsilon > 0$. Set $a:= \lim_\omega \frac{\|a_n\|}{s_n}$, where the limit is taken with respect to the ultrafilter $\omega$. Then $a=\|g\| > 0$. By definition of the ultralimit
\begin{align*}
\left \{ n \in \mathbb{N} \hspace{2mm} : \hspace{2mm} \left | \frac{\|a_n\|}{s_n}- a \right | < \epsilon \right \} \in \omega.
\end{align*}
Since filters are closed under taking supersets, this implies for $\epsilon=a/2$ using the inverse triangle inequality $ a -  \frac{\|a_n\|}{s_n} < \left | \frac{\|a_n\|}{s_n}- a \right |$ that
\begin{align*}
A:= \left \{ n \in \mathbb{N} \hspace{2mm} : \hspace{2mm} \|a_n\| > s_n(a- \epsilon)= \frac{s_na}{2} \right \} \in \omega.
\end{align*}
For $n \in A$ set $g_n=a_n$ and for $n \notin A$ define $g_n$ to be some element in $G$ satisfying $\|g\| \in (\frac{s_na}{2}, \frac{s_na}{2}+1]$, which exists since the norm $\|.\|$ is an unbounded word norm since $\Cone_\omega(G,s_n)$ possesses a non-trivial element. Then all $g_n$ satisfy $\|g_n\| \geq Ks_n$ with $K=a/2$. Moreover, $(g_n)_n$ defines an admissible sequence which still represents $g$ in $\Cone_\omega(G,s_n)$ since by definition $(g_n)_n$ and $(a_n)_n$ agree on $A \in \omega$. 
\end{proof}

\begin{lemma}\label{effectiveupperbound}
Let $G$ be a group and let $\|.\|$ be a conjugation invariant norm on $G$. Then for all $\delta >0$ and $g \in \Cone_\omega(G, \|.\|)$ there is a sequence $(g_n)_{n \in \N} \subset G$ with $g = [ (g_n)_n] $ and $\|g_n \| < (\|g \| + \delta)s_n $, where $\|g\|$ is the induced norm on $\Cone_\omega(G, \|.\|)$.  
\end{lemma}

\begin{proof}
This is similiar to the proof of Lemma \ref{2} above, but we include it for completeness. Let $(a_n)_{n \in \N}$ be a sequence representing $g \in \Cone_\omega(G, \|.\|)$. Set $a= \lim_\omega \frac{\|a_n\|}{n}=\|g \|$ with respect to the ultrafilter $\omega$. Then for all $\epsilon > 0 $ by definition of the ultralimit 
\begin{align*}
\left \{ n \in \mathbb{N} \hspace{2mm} : \hspace{2mm} \left | \frac{\|a_n\|}{s_n}- a \right | < \epsilon \right \} \in \omega.
\end{align*}
Since filters are closed under taking supersets, this implies setting $\epsilon = \delta$ that
\begin{align*}
A:= \left \{ n \in \mathbb{N} \hspace{2mm} : \hspace{2mm} \|a_n\| < s_n(\delta + a) \right \} \in \omega.
\end{align*}
Define $g_n=a_n $ for all $n \in A$ and $g_n=1 \in G$ for all $n \notin A$. Since $A \in \omega$ and $(g_n)_n$ agrees with $(a_n)_n$ on $A$, it follows that $[(g_n)_n] =g$. 
\end{proof}

\subsection{A bad example for a word norm generated by infinitely many conjugacy classes on the integers}

Let us construct an example of a word norm on $\Z$ generated by infinitely many conjugacy classes that despite remaining unbounded behaves very differently from the standard absolute value which is the word norm generated by $\{ \pm 1\}$. Recall, that $\Cone_\omega(Z, \|.\|) = \R$ for all norms $\|.\|$ generated by finitely many conjugacy classes by Lemma \ref{equiva}. 

Consider the word norm $\|.\|_S$ on $\Z$ generated by the set 
\[
S= \{ \pm 2^m m! \mid m \in \Z_{\geq 0} \}.
\]

\begin{claim}
Let $(x_n)_{n \in \N}$ be the sequence $x_n= 2^{n-1} n!$ for $n \geq 1$ and $x_0=0$. Then $\|x_n \|_S=n$ for all $n \in \N$. 
\end{claim}

Clearly, $x_n = n ( 2^{n-1} (n-1)!)$ and so $\| x_n \|_S \leq n$. Now assume that $x_n=s_1 + \dots + s_\ell$ is a presentation of $x_n$ of minimal length $\ell$ where $s_i \in S$ for all $i$. Reordering this expression we can assume that there is $k \in \{0, \dots, \ell \}$ such that the absolute values of the generators satisfy
\[
|s_i| \leq 2^{n-1} (n-1)! \hspace{2mm} \forall i \leq k \hspace{2mm} \text{and} \hspace{2mm} |s_i| \geq 2^n n! \hspace{2mm} \forall i > k.
\] 
Note that for all $i >k $ the generators are of the form $(2^n n!)\bar{s}_i$ for a positive integer $\bar{s}_i$. Thus, we calculate 
\[
\left | s_1 + \dots + s_k \right | = \left | x_n - (s_{k+1} + \dots + s_\ell) \right | =   \left | 2^{n-1} n! - (2^n n!)(\bar{s}_{k+1} + \dots + \bar{s}_\ell) \right | \geq 2^{n-1} n!.
\]
However, since $|s_i| \leq 2^{n-1} (n-1)! \hspace{2mm} \forall i \leq k$, it follows that $\left | s_1 + \dots + s_k \right | \leq k \cdot 2^{n-1} (n-1)! $. Therefore, $k \geq n$, which implies that $\ell = n$ since $\ell$ was the minimal length of an expression of $x_n$ in generators to begin with. This establishes the claim.

Therefore, the sequence $(x_n)_{n \in \N}$ defines a non-trivial element in $\Cone_\omega(\Z, \|.\|_S)$ for which we have $\|2x_n \|_S = \|2^n n!\|=1$ for all $n$. Thus, we have seen that $(x_n)_{n \in \N}$ represents a non-trivial element of order 2 in $\Cone_\omega(\Z, \|.\|_S)$.

\begin{theorem} \label{Z-torsion}
The asymptotic cone $\Cone_\omega(\Z, \|.\|_S)$ contains non-trivial elements of order two. In particular, $\Cone_\omega(Z, \|.\|_S)$ and $\R$ are not isomorphic. 
\end{theorem} 

\begin{remark}
A similar argument shows that for all $t \in \N$, $t \geq 2$ there exists an infinite generating set such $S_t$ such that $\Cone_\omega(\Z, \|.\|_{S_t})$ has non-trivial elements which are $t$-torsion. 
\end{remark} 

\subsection{How conjgation invariant norms and finitely generated word norms have trivial intersection}

This subsection is supposed to remind us that we cannot hope for interesting asymptotic behaviour whenever the generating set of our conjugation invariant word norm happens to be finite. 

\begin{lemma}
Let $G$ be a group generated by a finite set $S \subset G$ that is invariant under conjugation. Then $[G,G]$ is finite. 
\end{lemma}

\begin{proof}
Let $\pi \colon G \rightarrow \Sigma_{|S|}$ be the homomorphism induced by the action via conjugation as permutations of the set $S$. The kernel $\ker(\pi)=Z(G)$ is the center of $G$. We deduce that the center is a finite index subgroup. Hence, it is finitely generated itself and contains a finite index subgroup $Z$ which is free abelian of some rank $k$ and presents $G$ as an extension of a finite by a free abelian group
\begin{figure}[H]
\centering
\begin{tikzcd}
1 \arrow[-latex]{r} & Z \arrow[-latex]{r}{\iota}
 & G \arrow[-latex]{r}&  G/Z \arrow[-latex]{r} & 1.
\end{tikzcd}
\end{figure} 
The extension class of this extension is an element in $H^2(G/Z,Z)$, where the action on the coefficients is trivial. Let $m$ denote the cardinality of the group $G/Z$. Every element $c \in H^2(G/Z,Z)$ is torsion of order $m$ by \cite[Thm. 6.5.8]{Weibel}.

The group $\widetilde{G}= \frac{\Z^k \times G}{Z}$ where $Z$ is embedded diagonally via multiplication by $m$ and $\iota$ is an extension of $G/Z$ by $\Z^k \cong Z$ since  $\widetilde{G}/\Z^k= G/Z$. We obtain the following commutative diagram
\begin{figure}[H]
\centering
\begin{tikzcd}
&&&& 1 \arrow[-latex]{r} & Z \arrow[-latex]{r}{\iota} \arrow[-latex]{d}{\cdot m}
 & G \arrow[-latex]{r} \arrow[-latex]{d} &  G/Z \arrow[-latex]{r} \arrow[-latex]{d}{id} & 1 \\
&&&& 1 \arrow[-latex]{r} & \mathbb{Z}^k \arrow[-latex]{r}
 & \widetilde{G} \arrow[-latex]{r}&  G/Z \arrow[-latex]{r} & 1. & & (\star)
\end{tikzcd}
\end{figure}

By the naturality of the identification of isomorphism classes of extensions of $G/Z$ by $Z$ with elements of $H^2(G, G/Z)$ the extension class of $(\star)$ is given by the image of $c$ under the change of coefficients homomorphism induced by multiplication by $m$, which vanishes. Therefore, $(\star)$ is the trivial central extension which implies that $\widetilde{G} \cong \Z^k \times G/Z$. Thus,
\begin{align*}
[G,G] \subset [ \widetilde{G}, \widetilde{G}] = [ \Z^k \times G/Z,  \Z^k \times G/Z] = [G/Z,G/Z]
\end{align*}
and since the latter is finite, so is $[G,G]$. 
\end{proof}

\begin{theorem}
Let $G$ be a group generated by a finite set $S \subset G$ which is invariant under conjugation. Then $G$ is quasi-isometric to $\Z^n$ where $n = \rk(G^{ab})$.
\end{theorem}

\begin{proof}
Let $f \colon G \to G^{ab}$ be the abelianisation map and equip $G^{ab}$ with the word norm $\|.\|_{f(S)}$ generated by the set $f(S)$, which is equivalent to the standard norm on $G^{ab}$ since $f(S)$ is a finite generating set of $G^{ab}$. It is clear, that $\|f(g) \|_{f(S)} \leq \|g\|_S$ for all $g \in G$ and hence $f$ is Lipschitz continuous with constant one with respect to the norms $\|.\|_S$ and $\|.\|_{f(S)}$. Conversely, let $g \in G$ be arbitrary and $\|f(g)\|_{f(S)}=n$. Write $f(g)=f(s_1) \cdots f(s_n)$ and define $z = s_1 \cdots s_n \in G$. Then $\|z\|_S \leq n$ by definition of the norm, $f(z)=g$ and since $f$ is Lipschitz with constant 1 it follows that $\|z\|_s=n = \|f(g)\|_{f(S)}$. Moreover, $gz^{-1}= a \in \ker(f)=[G,G]$, which is finite. So, $\|g\|_S = \|az\|_S \leq \|a \|_S + \|z\|_S \|a\|_S + \|f(g)\|_{f(S)}$. Consequently, 
\begin{align*}
\|f(g) \|_{f(S)} \geq \| g \|_S - \max_{a \in [G,G]}\|a\|_S \hspace{2mm} \text{for all } g \in G.
\end{align*}
Since $f$ is surjective, $f$ is a quasi-isometry. 
\end{proof}

\section{Short exact sequences}

To keep the notation is this subsection simple we will set the scaling sequence to be $s_n=n$ and omit it from our notation.  

\begin{defi}
Let $f\colon G \rightarrow H$ be a surjective homomorphism and $\|.\|_S$ be a conjugation invariant word norm on $G$ generated by $S \subset G$. Then we call the conjugation invariant word norm $\|.\|_{f(S)}$ on $H$ generated by $f(S)$ the norm \textit{induced by} $\|.\|_S$ \textit{via} $f$. 
\end{defi}

It is immediate that in the setting of the above definition the map $f$ is Lipschitz-continuous with constant one with respect to the norms $\|.\|_S$ and $\|.\|_{f(S)}$. 

\begin{lemma}[SES-Lemma]
\label{SES-Lemma}
Let $ 1 \rightarrow H \xrightarrow{f} G \xrightarrow{g} G/H \rightarrow 1$ be a short exact sequence of groups where $H$ and $G$ are equipped with conjugation invariant norms such that $\|.\|_G$ is a word norm generated by a set $S \subset G$ and $f$ is a quasi-isometric embedding. Moreover, assume that the norm $\|.\|_{G/H}$ on $G/H$ is the norm induced by $S$ via $g$. Then the short exact sequence gives rise to a short exact sequence of asymptotic cones
\begin{equation} \label{cses}
\centering
\begin{tikzcd}
1  \arrow[-latex]{r} & \Cone_\omega(H,\|.\|_H) \arrow[-latex]{r}{\widehat{f}} 
 & \Cone_\omega(G,\|.\|_G) \arrow[-latex]{r}{\widehat{g}} & \Cone_\omega( G/H , \|.\|_{G/H} ) \arrow[-latex]{r} & 1.
\end{tikzcd}
\end{equation}
\end{lemma}

\begin{proof}
By Proposition \ref{qi-invariant} $\widehat{f}$ exists and is injective. Moreover, $\widehat{g}$ is well defined since $g$ is Lipschitz-continuous with constant one. Since $g \circ f $ is trivial, the same holds true for $\widehat{g} \circ \widehat{f}$. 

As a set $G$ is the disjoint union of its left $H$-cosets $G= \coprod_{i \in I} a_i H$. Since the norm $\|.\|_G$ only assumes discrete values, we can choose a set of coset representatives $\lbrace a_i \rbrace_{i \in I} $ of smallest norm. That is, $\|a_i\|_G \leq \| x\|_G $ for all $x \in a_i H$. 

It holds that $\|a_i\|_G = \|g(a_i)\|_{G/H}$. Indeed, since $g$ is Lipschitz-continuous with constant one, it holds that $\|g(a_i)\|_{G/H} \leq \|a_i\|_G $. However, if this inequality was strict, then $a_i$ was not the coset representative of smallest norm to begin with. To prove surjectivity let $x=[(x_n)_{n \in \mathbb{N}}] \in \Cone_\omega( G/H , \|.\|_{G/H} )$ be given. Then for each $n$ there exists $i(n) \in I$ with $g(a_{i(n)}) = x_n$. Since $\|a_{i(n)}\|_G = \|x_n\|_{G/H}$ the sequence $(a_{i(n)})_{n \in \mathbb{N}} $ is admissible and defines an element in  $\Cone_\omega(G,\|.\|_G)$ which satisfies $\widehat{g}([(a_{i(n)})_{n \in \mathbb{N}}])= [x_n]$.

It remains to show $\ker(\widehat{g}) \subset \im(\widehat{f})$. Let $x= [x_n] \in \ker(\widehat{g})$. Then, $\lim_\omega \frac{1}{n} \|g(x_n)\|_{G/H}= 0$. Write $x_n= a_{i(n)}h_n$ in the coset decomposition for $n \in \mathbb{N}$. Then  
\begin{align*}
\lim_\omega \frac{d_G(x_n,h_n)}{n}=\lim_\omega \frac{\|a_i(n)\|_G}{n} = \lim_\omega \frac{\|g(x_n)\|_{G/H}}{n}=0.
\end{align*}
Moreover, $\|h_n\|_G=\|x_n a_{i(n)}^{-1}\| \leq \|x_n\|_G + \|a_{i(n)}\|_G \leq 2 \|x_n\|_G$ since $a_{i(n)}$ is the element of smallest norm inside the left $H$-coset that $x_n$ belongs to. Thus, the sequence $(h_n)_{n \in N}$ represents an element in the asymptotic cone with $x=[x_n]=[h_n] \in \im(\widehat{f})$.
\end{proof}

If we only consider word norms generated by finitely many conjugacy classes, the assumption that the quotient carries the induced norm is not necessary. Namely, since any two word norms $\nu_R$ and $\nu_S$ generated by finitely many conjugacy classes are equivalent by Lemma \ref{equiva} and the identity, which clearly preserves the group structure, gives rise to a bi-Lipschitz isomorphism of the topological groups $\Cone_\omega(G/H, \nu_R)$ and $\Cone_\omega(G/H, \nu_S)$ by Proposition \ref{qi-invariant}. Then Lemma \ref{SES-Lemma} becomes:

\begin{lemma}\label{SES-fcg}
Let $ 1 \rightarrow H \xrightarrow{f} G \xrightarrow{g} G/H \rightarrow 1$ be a short exact sequence of groups and let $\|.\|_G$ and $\|.\|_{G/H}$ be conjugation invariant word norms generated by finitely many conjugacy classes. If $f$ is a quasi-isometric embedding, then there is the following short exact sequence on the level of asymptotic cones
\begin{equation*} 
\centering
\begin{tikzcd}
1  \arrow[-latex]{r} & \Cone_\omega(H,\|.\|_H) \arrow[-latex]{r}{\widehat{f}} 
 & \Cone_\omega(G,\|.\|_G) \arrow[-latex]{r}{\widehat{g}} & \Cone_\omega( G/H , \|.\|_{G/H}) \arrow[-latex]{r} & 1.
\end{tikzcd}
\end{equation*}
\end{lemma}

\begin{cor} \label{extentionR}
Let $G$ be an extension $\Z \to G \to K$ of a bounded group $K$ by $\mathbb{Z}$ and assume that $G$ is finitely normally generated.  
Let $\|.\|$ be a word norm generated by finitely many conjugacy classes on $G$ and assume that there exists quasimorphism on $G$ that is unbounded. Then $\Cone_\omega(G,\|.\|) \cong \mathbb{R}$ are bi-Lipschitz isomorphic as metric groups.
\end{cor}

\begin{proof}
By Lemma \ref{HomogenisationQM} we can assume that $\psi$ is an unbounded quasimorphism on $G$ that is homogeneous.  If $\psi_{|\Z}$ is bounded, then $\psi$ descends to a homogeneous quasimorphism $\bar	{\psi}$ by \cite[Lemma 3.16]{HartnickSchweitzer}. However, since $G$ is finitely normally generated so is $K$ by some set $S$. Then $\bar{\psi}$ is bounded on $S$. Thus, by Lemma \ref{Autqm implies Autnorm} the norm $\|.\|_S$ is stably unbounded if $\bar{\psi}$ is unbounded. This contradicts our assumption that $K$ is bounded. 
Therefore, $\psi_{|\Z}$ is an unbounded homogeneous quasimorphism on $\Z$. Since $\|.\|$ is generated by finitely many conjugacy classes, the restriction of $\|.\|$ onto $\Z$ is stably unbounded and thus equivalent to the standard absolute value on $\Z$. Then the result follows from Lemma \ref{SES-fcg} above.  
\end{proof}

\subsection{Applications}

According to \cite[Thm. 3.1]{KLM} all semisimple Lie groups are finitely normally generated.

\begin{theorem}\label{PropSemisimple}
Let $G$ be a connected semisimple Lie group with center $Z(G)$ isomorphic to the infinite cyclic group $\Z$ and let $\|.\|$ be a word norm generated by finitely many conjugacy classes on $G$. Then $\Cone_\omega(G,\|.\|) \cong \R$. 
\end{theorem}

\begin{proof}
Since $Z(G)$ is infinite, there exists a quasimorphism on $G$ which is unbounded on $Z(G)$ according to \cite[Prop.3.8]{KLM}. The quotient $G/Z(G)$ is a semisimple Lie group with trivial center and so bounded according to \cite[Thm 3.1]{KLM}. Then apply Corollary \ref{extentionR} to the central extension $Z(G) \rightarrow G \rightarrow G/{Z(G)}$ and obtain that the inclusion of the center into $G$ induces the bi-Lipschitz isomorphism of metric groups $\Cone_\omega(G,\|.\|) \cong \mathbb{R}$.
\end{proof}

\begin{example}
The universal cover of the modular group $\PSL(2,\mathbb{R})$ yields the central extension 
\begin{equation*}
\centering
\begin{tikzcd}
1  \arrow[-latex]{r} & \mathbb{Z} \arrow[-latex]{r} 
 & \widetilde{\PSL(2, \mathbb{R})} \arrow[-latex]{r} & \PSL(2,\mathbb{R}) \arrow[-latex]{r} & 1.
\end{tikzcd}
\end{equation*}
The group $\PSL(2,\mathbb{R})$ is bounded \cite[Prop 3.7]{KLM} and there is an unbounded quasimorphism, called the translation number, whose restriction to $\mathbb{Z}$ is unbounded as well \cite{Ghys}. The translation number arises as the lift of the famous rotation number of a homeomorphism of the circle. Therefore, for any norm $\|.\|$ on $\widetilde{\PSL(2, \mathbb{R})}$ generated by finitely many conjugacy classes the inclusion induces a bi-Lipschitz isomorphism  $\Cone_\omega(\widetilde{\PSL(2, \mathbb{R})}, \|.\|) \cong \mathbb{R}$ by Proposition \ref{PropSemisimple}. 

\end{example}

\begin{example}
The universal cover of the group $\Diff_0(S^1)$ of orientation preserving diffeomorphisms of the circle yields the central extension
\begin{equation*}
\centering
\begin{tikzcd}
1  \arrow[-latex]{r} & \mathbb{Z} \arrow[-latex]{r} 
 & \widetilde{\Diff_0(S^1)} \arrow[-latex]{r} & \Diff_0(S^1) \arrow[-latex]{r} & 1 .
\end{tikzcd}
\end{equation*}
According to \cite[Thm. 1.11]{Burago} the group $\Diff_0(S^1)$ is bounded. The translation number defines a quasimorphism $\widetilde{\Diff_0(S^1)} \rightarrow \mathbb{Z}$ which is unbounded on the image of $\mathbb{Z}$ in $\widetilde{\Diff_0(S^1)}$. It follows from Corollary \ref{extentionR} that for any norm $\|.\|$ on $\widetilde{\Diff_0(S^1)}$ generated by finitely many conjugacy classes $\Cone_\omega(\widetilde{\Diff_0(S^1)}, \|.\|)$ is bi-Lipschitz isomorphic to $\mathbb{R}$. 
\end{example}

\subsection{Topological Splittings}

\begin{defi}
Let $G$ be equipped with a word norm $\|.\|_G$ generated by finitely many conjugacy classes. Let $H \leq G$ be a normal subgroup and let $\|.\|_{G/H}$ the word norm on $g/H$ induced by $\|.\|$. A collection of coset representatives $\{a_i \}$ is called \textit{linearly bounded} if there exist positive constants $C,D$ such that  $\|a_i\|_G \leq C\|a_i\|_{G/H} + D$ for all $i$.
\end{defi}

%Maybe a lemma on its own that amissable collections yield well defined sections?

\begin{lemma}
\label{Serre}
Let the conditions be as in the SES-Lemma \ref{SES-Lemma}. Additionally assume that there exists a constant $C>0$ and a linearly bounded collection of coset representatives satisfying $\|a_ia_j^{-1}\|_G \leq C\| g(a_i) g(a_j)^{-1}\|_{G/H}$  for all $i,j$. 

Then the right section of sets $r \colon G/H \rightarrow G$ defined by $r(g(a_i))=a_i$ is Lipschitz. The induced map $\widehat{r} \colon \Cone_\omega( G/H , \|.\|_{G/H} ) \to \Cone_\omega(G,\|.\|_G)$ is a continuous right section of equation \eqref{cses} and there exists a homeomorphism 
\begin{align*}
\Cone_\omega(G,\|.\|_G) \cong \Cone_\omega(H,\|.\|_H) \times \Cone_\omega( G/H , \|.\|_{G/H}) .
\end{align*}   
\end{lemma}

\begin{proof}
The Lipschitz-continuity of the right section $r$ is immediate from the condition that $\|a_ia_j^{-1}\|_G \leq C\| g(a_i) g(a_j)^{-1}\|_{G/H}$  for all $i,j$.
Since the collection is linearly bounded, the induced map $\widehat{r} \colon \Cone_\omega(G/H,\|.\|_{G/H}) \rightarrow \Cone_\omega(G,\|.\|_G)$ is well-defined and the Lipschitz continuity of $\widehat{r}$ follows from the Lipschitz continuity of $r$. It remains to prove that $\Cone_\omega(G,\|.\|_G)$ is homeomorphic to the product of the two other asymptotic cones.

Every $b \in \Cone_\omega(G,\|.\|_G)$ can be written uniquely as a product $b=\widehat{f}(a) \widehat{r}(c)$ where $a \in \Cone_\omega(H,\|.\|_H)$ and  $c=\widehat{g}(b) \in \Cone_\omega(G/H, \|.\|_{G/H})$. That is, on the level of sets there is the well-defined map 
\begin{align*}
\psi \colon \Cone_\omega(G, \|.\|_G) & \rightarrow \Cone_\omega(H,\|.\|_H) \times \Cone_\omega ( G/H , \|.\|_{G/H} ) \\
b & \rightarrow (a,c)  
\end{align*}
On the second factor $\psi$ is just $\widehat{g}$, so it remains to check continuity of the map $b \rightarrow a$ in order to establish continuity of $\psi$ by the universal property of the product. Let $b_1= \widehat{f}(a_1)s(c_1)$ and $b_2= \widehat{f}(a_2)s(c_2)$ be two elements in $\Cone_\omega(G, \|.\|_G)$. Abuse notation to denote the norms induced norms on the asymptotic cones by $\|.\|_H$, $\|.\|_G$ and $\|.\|_{G/H}$ as well. Then
\begin{align*}
\|a_1a_2^{-1}\|_H & = \|\widehat{f}(a_1 a_2^{-1})\|_G  \\
& = \|\widehat{f}(a_2)^{-1} \widehat{f}(a_1)\|_G \\
& = \|\widehat{f}(a_2)^{-1} \widehat{f}(a_1) \widehat{r}(c_1) \widehat{r}(c_2)^{-1} \widehat{r}(c_2) \widehat{r}(c_1)^{-1} \|_G \\
& \leq \|\widehat{f}(a_2)^{-1} \widehat{f}(a_1) \widehat{r}(c_1) \widehat{r}(c_2)^{-1}\|_G + \|\widehat{r}(c_2) \widehat{r}(c_1)^{-1} \|_G \\
& = \|\widehat{r}(c_2)^{-1} \widehat{f}(a_2)^{-1} \widehat{f}(a_1) \widehat{r}(c_1)\|_G + \|\widehat{r}(g(b_2) \widehat{r}(g(b_1))^{-1} \|_G\\
& \leq \|b_2^{-1} b_1\|_G + C\|g(b_2) g(b_1)^{-1}\|_{G/H} \\
& \leq \|b_2^{-1} b_1\|_G + C\|b_2 b_1^{-1} \|_G \\
& = (1+C)\|b_1 b_2^{-1}\|_G ,
\end{align*}
which shows Lipschitz continuity of the map $b \rightarrow a$. Define 
\[
\phi\colon \Cone_\omega(H,\|.\|) \times \Cone_\omega(G/H, \|.\|_{G/H} ) \rightarrow \Cone_\omega(G, \|.\|)
 \hspace{2mm}\text{ by } \hspace{2mm} \phi(a,c)= \widehat{f}(a) \widehat{r}(c).
 \]
 Then continuity of $\phi$ follows since $\phi = m \circ ( \widehat{f} \times \widehat{r})$ where $m$ is the continuous multiplication map on $\Cone_\omega(G, \|.\|)$. Clearly, we have that $\psi$ and $\phi$ are inverse to one another which finishes the proof.
\end{proof}

\begin{lemma}
\label{fcg}
Let $G$ be a group together with a word norm $\|.\|_G$ generated by finitely many conjugacy classes. Then for any word norm $\|.\|_{G^{ab}}$ generated by finitely many conjugacy classes on the abelianisation $G^{ab}=G/[G,G]$ there exists $C \geq 0$ such that a system of coset representatives $\{a_i\}$ of $[G,G]$ in $G$ satisfies  $\|a_ia_j^{-1}\|_G \leq C\|p(a_i) p(a_j)^{-1}\|_{G^{ab}}$ where $p \colon G \to G^{ab}$ denotes the projection.
\end{lemma}

\begin{proof}
$G$ is generated by finitely many conjugacy classes $Conj(g_1),...,Conj(g_n)$ and so $G^{ab}$ is generated by images of those. Then $G^{ab}$ is a finitely generated abelian group and by the classification theorem $G^{ab} \cong \mathbb{Z}^k \oplus T$, where $k \leq n$ and $T$ is a finite abelian group.  Let $g_1^\prime, ..., g_k^\prime$ be preimages under $p$ of the standard generators of $\mathbb{Z}^k$ in $G^{ab}$ and let $f_t$, $t \in T$ be a collection of preimages of the torsion factor $T$ in $G^{ab}$. 

Then the conjugacy classes of the finite collection $g_1,...,g_n, g_1^\prime , ... , g_k^\prime , f_t$ generate $G$ and define a word norm $\|.\|$ on $G$. Then $\|.\|_G$ and $\|.\|$ on $G$ are Lipschitz equivalent on $G$ by Lemma \ref{equiva}. Also, the finitely generated word norm $\|.\|_{G^ab}$ on $G^{ab} \cong \mathbb{Z}^k \oplus T$ is Lipschitz equivalent to the finitely generated word norm $\|.\|_{std}$ generated by the standard generators of $\mathbb{Z}^k$, their inverses and every element of $T$. 

Let $x=(l,t) \in \mathbb{Z}^k \oplus T \cong G^{ab}$.  Set the representative $a_x$ to be $a_x= (g_1^\prime)^{l_1}...(g_k^\prime)^{l_k} f_t$. For $x=(l,t)$ and $y=(m,s)$ it holds that $\|xy^{-1}\|_{std}=\sum_{i=1}^k |l_i - m_i| + d(t,s)$, where $d$ is the discrete metric. Then
\begin{align*}
\|a_xa_y^{-1}\| & = \|(g_1^\prime)^{l_1}...(g_k^\prime)^{l_k} f_t \Big((g_1^\prime)^{m_1}...(g_k^\prime)^{m_k} f_s \Big )^{-1} \| \\
& = \|(g_1^\prime)^{l_1-m_1} (g_2^\prime)^{l_2}...(g_k^\prime)^{l_k} f_t f_s^{-1} (g_k^\prime)^{-m_k} ...   (g_2^\prime)^{-m_2}\| \\
& \leq \|(g_1^\prime)^{l_1-m_1}\|+ \|(g_2^\prime)^{l_2-m_2} (g_3^\prime)^{l_3}...(g_k^\prime)^{l_k} f_t f_s^{-1} (g_k^\prime)^{-m_k} ...   (g_3^\prime)^{-m_3}\| \\
& \leq \dots \leq \sum_{i=1}^k \|(g_i^\prime)^{l_i- m_i}\| + \|f_tf_s^{-1}\| \\
& \leq \sum_{i=1}^k |l_i - m_i|+ \|f_tf_s^{-1}\|\\
& = \|xy^{-1}\|_{std} - d(t,s) + \|f_tf_s^{-1}\| .
\end{align*}
Since $T$ is finite, there exists a constant $K> 0 $ such that $\|f_tf_s^{-1}\|-d(t,s) \leq K$  for all $t,s \in T$. Let $L_1 >0 $ denote the Lipschitz constant such that $ \|.\|_G \leq L_1 \|.\|$ and $L_2 > 0$ the Lipschitz constant such that $\|.\|_{std} \leq L_2\|.\|_{G^{ab}}$. Then the above calculation yields 
\begin{align*}
\|a_xa_y^{-1}\|_G \leq L_1 \|a_xa_y^{-1}\| \leq L_1 \big( \|xy^{-1}\|_{std}+K \big) \leq L_1 L_2 \|xy^{-1}\|_{G^{ab}}+ L_1 K. 
\end{align*}
Thus, for all $x \neq y$ it holds that $\|a_xa_y^{-1}\|_G \leq (L_1 L_2 + L_1 K) \|xy^{-1}\|_{G^{ab}}$ and the result follows.
\end{proof}

\begin{theorem}\label{producthomeo}
Let $G$ be a group. Let $\|.\|$ and $\|.\|_{G^{ab}}$ be word norms generated by finitely many conjugacy classes on $G$ and its abelianisation $G^{ab}$. Then there are the homeomorphisms 
\begin{align*}
\Cone_\omega(G,\|.\|) \cong \Cone_\omega([G,G],\|.\|) \times \Cone_\omega(G^{ab} , \|.\|_{ab}) \cong \Cone_\omega([G,G],\|.\|) \times (\R^k, \ell_1) 
\end{align*} 
where $k=\rk(G^{ab})$. 
\end{theorem}

\begin{proof}
Since $G$ is finitely conjugation generated, $G^{ab} \cong \mathbb{Z}^k \oplus T$ for a finite torsion group $T$. By Lemma \ref{fcg} the induced norm on $G^{ab}$ is equivalent to the standard word norm and therefore $\Cone_\omega(G,\|.\|_{ab}) \cong \Cone_\omega(\mathbb{Z}^k \oplus T, \|.\|_{std})$ is Lipschitz-isomorphic to $( \mathbb{R}^k, \ell_1)$ by Proposition \ref{coarseequiv}.
By Lemma \ref{SES-fcg} there exists a short exact sequence on the level of asymptotic cones 
\begin{equation*}
\centering
\begin{tikzcd}
1  \arrow[-latex]{r} & \Cone_\omega([G,G],\|.\|) \arrow[-latex]{r} 
 & \Cone_\omega(G,\|.\|) \arrow[-latex]{r} & \Cone_\omega( G^{ab} , \|.\|_{ab} ) \arrow[-latex]{r} & 1.
\end{tikzcd}
\end{equation*}
By Lemma \ref{Serre} and Lemma \ref{fcg} this sequence is topologically right split and
gives rise to a homeomorphism $\Cone_\omega(G,\|.\|) \cong \Cone_\omega([G,G],\|.\|) \times \Cone_\omega(G^{ab} , \|.\|_{ab})$. 
\end{proof}

\noindent
The contraction of $( \mathbb{R}^n, \ell_1)$ to the origin is a deformation retraction of $\Cone_\omega(G,\|.\|)$ onto $\Cone_\omega([G,G],\|.\|) \times  \lbrace 0 \rbrace$.

\begin{cor}
The inclusion $\Cone_\omega([G,G],\|.\|) \hookrightarrow \Cone_\omega(G,\|.\|)$ is a deformation retraction.
\end{cor} \qed

\begin{lemma}\label{qmshortexseq}
Let $G$ be a group together with a norm $\|.\|$ generated by finitely many conjugacy classes.
Let $\psi \colon G \to \R$ be a non-trivial homogeneous quasimorphism.  Then:
\begin{enumerate}[label=(\roman*)]
\item There is an induced short exact sequence of groups
\begin{equation} \label{qmses}
\centering 
\begin{tikzcd}
1  \arrow[-latex]{r} & \ker \widehat{\psi} \arrow[-latex]{r} 
 & \Cone_\omega(G,\|.\|) \arrow[-latex]{r}{\widehat{\psi}} & \mathbb{R} \arrow[-latex]{r} & 1 ,
\end{tikzcd}
\end{equation}
\item There exists a Lipschitz continuous right section $s\colon \mathbb{R} \rightarrow \Cone_\omega(G,\|.\|)$ that is a group homomorphism ,
\item The short exact sequence \eqref{qmses} is topologically split. That is, there exsits a homeomorphism $\Cone_\omega(G,\|.\|) \cong \ker \widehat{ \psi} \times \mathbb{R}$ .
\end{enumerate}
\end{lemma}

\begin{proof}
\textit{(i)} \hspace{1mm} Since $\psi$ is a quasimorphism bounded on a generating set of $\|.\|$, it induces a homomorphism $\widehat{\psi} \colon \Cone_\omega(G,\|.\|) \to \Cone_\omega(\Z, |.|)= \R$. Since $\psi$ is non-trivial, there exists $g \in G$ such that $\psi(g) \neq 0$. Assume for simplicity that $\psi(g)=1$ since the general case follows by rescaling.

Let $x=[(x_n)_{n \in \mathbb{N}}] \in \Cone_\omega(\Z, |.|) = \mathbb{R}$ be arbitrary. Set  $y_n=g^{x_n}$. Then $\|y_n\| \leq \|g\| \cdot |x_n|$ and so $(y_n)$ is an admissible sequence since $(x_n)$ is. Then 
\begin{align*}
\widehat{\psi}(y)=[(\psi(y_n))_{n \in \mathbb{N}}]= [(\psi(g^{x_n})_{n \in \mathbb{N}}]= [(x_n\psi(g))_{n \in \mathbb{N}}] = [(x_n)_{n \in \mathbb{N}}]=x,
\end{align*}
which shows that $\widehat{\psi}$ is surjective.

\textit{(ii)} \hspace{1mm} Define $s^\prime \colon \mathbb{Z} \rightarrow G$ by $s^\prime(n)=g^n$. By definition $s^\prime$ is a homomorphism and satisfies $\|s^\prime(n)\| \leq \|g\| \cdot |n| $ for all $n \in \Z$. Therefore, $s^\prime$ induces a group homomorphism $s \colon \Cone_\omega(\mathbb{Z}, |.|) \rightarrow \Cone_\omega(G,\|.\|)$ given by $s \big( [x_n] \big) = [s^\prime ( x_n)]= [g^{x_n}]$. By construction $s$ is a right inverse of the homomorphism $\widehat{\psi} \colon \Cone_\omega(G,\|.\|) \to \Cone_\omega(\Z, |.|)$.

\textit{(iii)} \hspace{1mm} The product homeomorphism is constructed analogously to the one in Lemma \ref{Serre}. 
\end{proof}

\begin{cor}
The kernel $\ker \widehat{\psi}$ arising from any quasimorphism $\psi \colon G \to \R$ is a deformation retract of $\Cone_\omega(G, \|.\|)$ whenever $\|.\|$ is a word norm generated by finitely many conjugacy classes. \qed
\end{cor}

\section{General Contraction Lemma for a single group}

In the proof of the continuity of our contracting homotopy on the asymptotic cone in Lemma \ref{ContractionLemma} we will use the following statement from calculus to derive the continuity of a map from a product space from continuity on the individual factors.

\begin{lemma}\label{continuity on factors}
Let $X,Y,Z$ be metric spaces together with a map $F \colon X \times Y \rightarrow Z$. Assume that for any $x \in X$ there is a constant $L_x$ such that $F(x,-) \colon Y \rightarrow Z$ is Lipschitz continuous with constant $L_x$, the map assigning $L_x$ to $x$ is continuous and further assume that for all $y \in Y$ the map $F(-,y) \colon X \rightarrow Z$ is continuous. Then $F$ is continuous on $X \times Y$. 
\end{lemma}

\begin{proof}
Let $(x,y) \in X \times Y$ and let $\epsilon > 0$ be given. Since the assignment of Lipschitz constants is continuous, there exists $\delta_0 > 0$ such that $|L_{x^\prime}- L_x| < 1$ for all $x^\prime$ with $d_X (x,x^\prime) < \delta_0$. This implies $L_{x^\prime} < L_x+1$ for all such $x^\prime$. By continuity of $F(-,y)$ there exists $\delta_1 > 0$ such that for all $x^\prime$ with $d(x,x^\prime) < \delta_1$ we have $d_Z(F(x,y),F(x^\prime,y)) < \epsilon/2$.  For the product metric it holds for any $(x^\prime, y^\prime)$ that $d_{X}(x,x^\prime) \leq d_{X \times Y} ((x,y), (x^\prime, y^\prime))$ and  $d_{Y}(y,y^\prime) \leq d_{X \times Y} ((x,y), (x^\prime, y^\prime))$. Choose $\delta = \min \{ \delta_0, \delta_1, \frac{\epsilon}{2(L_x+1)} \} >0 $. So for all $(x^\prime, y^\prime) \in B_{\delta}((x,y))$ we calculate
\begin{align*}
d_Z(F(x,y),F(x^\prime, y^\prime)) & \leq d_Z(F(x,y),F(x^\prime, y)) + d_Z (F(x^\prime,y),F(x^\prime, y^\prime)) \\
& < \epsilon/2 + L_{x^\prime} d_Y(y,y^\prime) \\
& < \epsilon/2 + (L_x+1)\delta \\
& \leq \epsilon/2 + \epsilon/2 = \epsilon.
\end{align*}
Therefore, $F$ is continuous.
\end{proof}

\begin{lemma}[Contraction Lemma]\label{ContractionLemma}
Let $G$ be a group together with a norm $\|.\|$ that induces the metric $d$. Assume there exist constants $K,L$ and projection maps $p_k \colon G \rightarrow G$ for all $k \in \N$ with $p_0=id$ such that 
\begin{enumerate}[label=(\roman*)]
\item $d(p_k(g),p_k(h)) \leq K d(g,h)$ for all $k \in \N$, $g,h \in G$,
\item $d(p_k(g),p_\ell (g)) \leq L |k- \ell |$ for all $k, \ell \in \N$, $g \in G$, 
\item $\| p_k(g) \| \leq \Big | \|g\|-k \Big |$ for all $k \in \N$, $g \in G$. 
\end{enumerate}
Then all asymptotic cones of $G$ are contractible. That is, $\Cone_\omega(G,s_n)$ is contractible for all ultrafilters $\omega$ and scaling sequences $s_n$. 
\end{lemma}

\begin{proof}
Denote the equivalence class of the identity sequence by $1 \in \Cone_\omega(G, s_n)$ and write $\|[g_n] \|= d([g_n],1)= \lim_\omega \frac{\|g_n\|}{s_n}$ for $[g_n] \in \Cone_\omega(G, s_n)$.
We want to use the projection maps to collapse a representative sequence with respect to a time parameter. For this we set our contracting homotopy to be 
\begin{align*}
H \colon & \Cone_\omega(G,s_n) \times [0,1]  \rightarrow \Cone_\omega(G,s_n),  \\
\big ( & [x_n], t \big)  \rightarrow [p_{\lambda_n}(x_n)] \hspace{3mm} \text{with} \hspace{1mm} \lambda_n \in \N \hspace{1mm} \text{such that} 
\begin{cases}      
 \lim_\omega \frac{\lambda_n}{\|x_n\|}= t & \text{if} \hspace{1mm} [x_n] \neq 1,\\
      \lambda_n=0 \hspace{1mm} \text{for all} \hspace{1mm} n & \text{if} \hspace{1mm} [x_n]=1.
\end{cases}
\end{align*}
\begin{claim}
$H$ is well defined. 
\end{claim}
First, we show that such a sequence $(\lambda_n)$ can be chosen. Recall, that for $x \in \Cone_\omega(G,s_n)$ with $x \neq 1$ any representative sequence is nonzero on an $\omega$-heavy subset. Namely, by definition of the ultralimit for all $\epsilon >0$ it holds that 
\begin{align*}
\left \{ n \in \mathbb{N} \hspace{2mm} : \hspace{2mm} \left | \frac{\|x_n\|}{s_n}- \|x\| \right | < \epsilon \right \} \in \omega.
\end{align*}
Since filters are closed under taking supersets, this implies by the inverse  triangle inequality and setting $\epsilon = \frac{\|x\|}{2}$ that
\begin{align*}
A := \left \{ n \in \mathbb{N} \hspace{2mm} : \hspace{2mm} \|x_n\| > s_n(\|x\|- \epsilon)= \frac{s_n\|x\|}{2} >0 \right \} \in \omega.
\end{align*}
This implies that any representative sequence of any non-trivial $x \in \Cone_\omega(G,s_n)$ admits for any $t \in [0,1]$ a choice of a positive real sequence $(\lambda_n^\prime)$  such that $\lim_\omega \frac{\lambda_n^\prime}{\|x_n\|}= t$. Since $\|x_n\| \rightarrow \infty$, we have for $\lambda_n= \lceil \lambda_n^\prime \rceil$ that $\lim_\omega \frac{\lambda_n}{\|x_n\|}= t$ as well.

Second, we verify the independence of the choice of such a sequence $(\lambda_n)$ while keeping the representative sequence $(x_n)$ of $x \neq 1$ fixed. Suppose $(\tau_n)$ is  a second sequence such that $\lim_\omega \frac{\tau_n}{\|x_n\|}= t = \lim_\omega \frac{\lambda_n}{\|x_n\|}$. Both limits $\lim_\omega \frac{\|x_n\|}{s_n}= \|x\|$ and $\lim_\omega \frac{s_n}{\|x_n\|}=\frac{1}{\|x\|}$  exist and belong to the open interval $(0, \infty)$. Then, calculate using the multiplicativity of ultralimits and condition $(ii)$ that 
\begin{align*}
d([p_{\lambda_n}(x_n)],[p_{\tau_n}(x_n)]) & = \lim_\omega \frac{d(p_{\lambda_n}(x_n),p_{\tau_n}(x_n))}{s_n} \leq \lim_\omega \frac{L|\lambda_n - \tau_n|}{s_n} \\
& = L \cdot \left (\lim_\omega \frac{|\lambda_n - \tau_n|}{\| x_n \|} \cdot \frac{\|x_n\|}{s_n} \right ) 
 = L \cdot \left | \lim_\omega \frac{\lambda_n - \tau_n}{\|x_n \|} \right | \cdot \|x\| = 0.
\end{align*}

Third, we verify the independence of $H$ of the choice of representative sequence $(x_n)$ of $x$ while keeping the sequence $\lambda_n$ fixed. For this let $(y_n)$ be another choice of representative sequence for $x \neq 1$. Then, calculate using condition $(i)$ that
\begin{align*}
d([p_{\lambda_n}(x_n)],[p_{\lambda_n}(y_n)]) = \lim_\omega \frac{d(p_{\lambda_n}(x_n),p_{\lambda_n}(y_n))}{s_n} \leq \lim_\omega \frac{Kd(x_n,y_n)}{s_n} = 0.
\end{align*}
This concludes the proof that $H$ is indeed well defined.  

Clearly, $H(-,0)=id_{\Cone_\omega(G,s_n)}$ using for example $\lambda_n =0$ for all $x \in \Cone_\omega(G,s_n)$ and  $n \in \N$. Moreover, condition $(iii)$ forces $H(-,1) \equiv 1$ using $\lambda_n = \lceil \| x_n \| \rceil$ for some representative sequence of $x$. It only remains to check continuity of $H$ to show that $H$ is the desired contracting homotopy.

\begin{claim}
$H$ is continuous. 
\end{claim}

Since $H$ is a map from a product of metric spaces it suffices by Lemma \ref{continuity on factors} to verify continuity in each argument separately and see that for each $x \in \Cone_\omega(G,s_n)$ the map $H(x,-)$ is Lipschitz with constant $L\|x\|$ since this is a continuous assignment. 

We start by checking continuity in the time parameter. By definition $H(1,-) \equiv 1$ is constant. Let $t,t^\prime \in [0,1]$ and $x=[x_n] \in \Cone_\omega(G,s_n)$ with $x \neq 1$. Let $(\lambda_n)$ be such that $\lim_\omega \frac{\lambda_n}{\|x_n\|}= t$ and $(\tau_n)$ be such that $\lim_\omega \frac{\tau_n}{\|x_n\|}= t^\prime$. Then
\begin{align*}
 d(H(x,t),H(x,t^\prime))= & d([p_{\lambda_n}(x_n)],[p_{\tau_n}(x_n)])= \lim_\omega \frac{d(p_{\lambda_n}(x_n),p_{\tau_n}(x_n))}{s_n} \leq \lim_\omega \frac{L|\lambda_n - \tau_n|}{s_n} \\
 = & L \lim_\omega \frac{|\lambda_n - \tau_n|}{\| x_n \|} \cdot \frac{\|x_n\|}{s_n}  = L \lim_\omega \frac{|\lambda_n - \tau_n |}{\|x_n\|} \cdot \lim_\omega \frac{\|x_n\|}{s_n}
 \\
 \leq & L \|x\| \cdot  |t-t^\prime| ,
\end{align*}
which establishes Lipschitz continuity with constant $L\|x\|$ for $H(x,-)$. 

Next we need to show for fixed $t$ continuity in $x \in \Cone_\omega(G,s_n)$. Since $H(-,1)$ is the identity map and $H(-,1)$ is constant we can assume that $t \in (0,1)$. Moreover, since the projection maps are all strictly norm decreasing the continuity at $1 \in \Cone_\omega(G,s_n)$ is immediate. So let $x=[x_n] \neq 1$ be given and let $y=[y_n]$ be a different element converging to $x$. 
We calculate for $(\lambda_n)$ and $(\tau_n)$ such that $\lim_\omega \frac{\tau_n}{\|y_n\|}= t = \lim_\omega \frac{\lambda_n}{\|x_n\|}$ the following
\begin{align*}
d(H(x,t),H(y,t)) & =
d([p_{\lambda_n}(x_n)],[p_{\tau_n}(y_n)]) \\
& = \lim_\omega \frac{d(p_{\lambda_n}(x_n),p_{\tau_n}(y_n))}{s_n} \\
& \leq \lim_\omega \frac{d(p_{\lambda_n}(x_n),p_{\tau_n}(x_n))}{s_n} + \frac{d(p_{\tau_n}(x_n),p_{\tau_n}(y_n))}{s_n} \\
& \leq \lim_\omega \frac{L|\lambda_n- \tau_n |}{s_n} + \lim_\omega \frac{Kd(x_n,y_n)}{s_n} \\
& = \lim_\omega \frac{L|\lambda_n- \tau_n |}{s_n} + Kd(x,y).
\end{align*}
Clearly $Kd(x,y) \xrightarrow{y \to x} 0$, so it remains to show that $\lim_\omega \frac{|\lambda_n- \tau_n |}{s_n} \xrightarrow{y \to x} 0$ as well. For this we first note that by continuity of the norm we have $\|y \| \xrightarrow{y \to x} \|x\| > 0$. Then 
\begin{align*}
\lim_\omega \frac{\tau_n}{\|x_n\|} = \lim_\omega \frac{\tau_n}{\|y_n\|} \cdot \frac{\|y_n \|}{\| x_n \|} = \lim_\omega \frac{\tau_n}{\|y_n\|} \cdot \lim_\omega \frac{\|y_n \|}{\| x_n \|} = t \cdot \lim_\omega \frac{\|y_n \|}{\| x_n \|} \xrightarrow{ y \to x} t. 
\end{align*}
Finally, we have
\begin{align*}
\lim_\omega  \frac{|\lambda_n- \tau_n |}{s_n} = \lim_\omega \frac{|\lambda_n- \tau_n |}{\|x_n\|} \cdot \frac{\|x_n\|}{s_n} \leq \left | \lim _\omega \frac{\lambda_n}{\|x_n\|} - \lim_\omega \frac{\tau_n}{\|x_n\|} \right | \cdot \|x\|  \xrightarrow{y \to x} \left | t - t \right | = 0,
\end{align*} 
which establishes $H(y,t) \xrightarrow{ y \to x} H(x,t)$ and concludes the proof. 
\end{proof}

\begin{lemma}\label{SingleprojectionContractionLemma}
Let $G$ be a group with a norm $\|.\|$ that induces the metric $d$. Then the conditions of the Contraction Lemma \ref{ContractionLemma} are fulfilled for $(G, \|.\|)$ if there is a single map $p \colon G \rightarrow G$ and a constant $L \in \N$ such that 
 \begin{enumerate}[label=(\roman*)]
\item $d(p(g),p(h)) \leq  d(g,h)$ for all $g,h \in G$,
\item $d(p(g),g) \leq L $ for all $g \in G$, 
\item $\| p(g) \| \leq \|g\|-1 $ for all $g \in G$ of norm $\|g\| \geq 1$,
\item $p(g)=1$, where $1 \in G$ is the trivial element, for all $g \in G$ of norm $\|g\| \leq 1$.
\end{enumerate}
\end{lemma}

\begin{proof}
We set $p_k$ to be the $k$-fold composition of $p$ for all $k \in \N$ using the convention that $p_0=id$. By induction it is immediate that the first condition $(i)$ implies the first condition of the Contraction Lemma \ref{ContractionLemma} with $K=1$. Similarly, the second condition $(ii)$ together with the triangle inequality implies by induction that $d(p_k(g),p_\ell (g)) \leq L |k- \ell |$ for all $k, \ell \in \N$ and for all $g \in G$. Finally, the third and fourth condition imply by induction that
\begin{align*}
\|p_k(g)\| \leq
\begin{cases}
\|g\|-k \hspace{2mm} \text{if} \hspace{2mm} \|g\| \geq k, \\
0 \hspace{12mm} \text{otherwise},
\end{cases}
\end{align*}
since $p_k(g)=1$ if $\|g\| \leq k$.
Thus, all conditions of the Contraction Lemma \ref{ContractionLemma} are satisfied.  
\end{proof}

\begin{remark}
Note, that if the norm $\|.\|$ is integer valued, the conditions $(iii)+(iv)$ are equivalent to assuming that $p(1)=1$ and $\|p(g)\| < \|g\|$ for all non-trivial $g \in G$. 
\end{remark}

\subsection{Direct sums}

Let $\{(G_i, \|.\|_i) \}_{i \in I}$ of groups together with norms on them for some indexing set $I$. Then the direct sum $\bigoplus_{i \in I} G_i$ can be equipped with the $\ell_1$-norm $\|.\| = \sum_i \|.\|_i$. Note, that in the case of a finite set $I$ this norm does induce the product topology on the direct sum. 

\begin{theorem}\label{directsumcontraction}
Let $I$ be any totally ordered set. Let $\{ (G_i,\|.\|_i, p_i) \}$ be a collection of groups together with integer valued norms $\|.\|_i$ and projections $p_i \colon G_i \rightarrow G_i$. Denote the induced metrics by $d_i$ and assume for all $i \in I$ the collection satisfies 
 \begin{enumerate}[label=(\roman*)]
\item $d_i(p(g),p(h)) \leq  d_i(g,h)$ for all $g,h \in G_i$,
\item $d_i(p_i(g),g) \leq 1 $ for all $g \in G_i$, 
\item $\| p_i(g) \|_i < \|g\|_i$ for all non-trivial $g \in G_i$,
\item $p_i(1)=1$, where $1 \in G_i$ is the trivial element.
\end{enumerate}
Then all asymptotic cones of $G= \bigoplus_{i \in I}$ with respect to the $\ell_1$-norm are contractible. 
\end{theorem}

\begin{proof}
Let $g \in G$ be non-trivial and write $g$ in its decomposition as an ordered direct sum $g=g_{i_1} + \dots + g_{i_k}$ for some $k \in \N$, where $i_1 < \dots < i_k$ and $g_{i_j} \in G_{i_j}$ are all non-trivial. Define 
\[
p \colon G \rightarrow G , \hspace{5mm} p(g)= p_{i_1}(g_{i_1})+ g_{i_2} + \dots + g_{i_k}
\]
and $p(1)=1$.

We will verify now that this projection satisfies the conditions in Lemma \ref{SingleprojectionContractionLemma}. By definition $p(1)=1$ and for all non-trivial $g \in G$ it holds that
\begin{align*}
d(p(g),g) &= d_{i_1}(p_{i_1}(g_{i_1}),g_{i_1}) + d_{i_2}(g_{i_2},g_{i_2}) + \dots + d_{i_k}(g_{i_k}, g_{i_k}) = d_{i_1}(p_{i_1}(g_{i_1}),g_{i_1}) \leq 1, \\
\|p(g)\| & = \|p_{i_1}(g_{i_1})\|_{i_1} + \|g_{i_2} \|_{i_2}+ \dots + \|g_{i_k} \|_{i_k}  \\
& < \|g_{i_1} \|_{i_1} + \|g_{i_2} \|_{i_2} + \dots + \|g_{i_k} \|_{i_k}  = \|g \| .
\end{align*}
It remains to verify that $d(p(g),p(h)) \leq d(g,h)$ for all $g,h \in G$. Write $g= g_{i_1} + \dots + g_{i_k}$ and $h= h_{j_1} + \dots + h_{j_\ell}$ in their ordered direct sum decomposition. If $i_1 = j_1$ then 
\begin{align*}
d(p(g),p(h)) & = d_{i_1}(p_{i_1}(g_{i_1}),p_{i_1}(h_{i_1})) + \sum_{i > i_1} d_i(g_i, h_i) \\
& \leq d_{i_1}(g_{i_1}, h_{i_1}) + \sum_{i > i_1} d_i(g_i, h_i) = d(g,h).
\end{align*}
If $i_1 \neq j_1$ then we can assume without loss of generality that $i_1 < j_1$. Then
\begin{align*}
d(g,h)= \|g_{i_1} \|+ \sum_{i > i_1} \|g_i h_i^{-1} \|_i .
\end{align*}
Consequently, we have 
\begin{align*}
d(p(g),h) = \|p_{i_1}(g_{i_1}) \|+ \sum_{i > i_1} \|g_i h_i^{-1} \|_i  \leq \|g_{i_1} \|+ \sum_{i > i_1} \|g_i h_i^{-1} \|_i  -1 =d(g,h)-1,
\end{align*}
and we conclude that $d(p(g),p(h)) \leq d(p(g),h) + d(h,p(h)) \leq d(g,h)-1 + 1 = d(g,h)$. 

Therefore, by Lemma \ref{SingleprojectionContractionLemma} the Contraction Lemma \ref{ContractionLemma} applies and every asymptotic Cone of $(G, \|.\|)$ is contractible.
\end{proof}

\begin{remark}
If one uses the axiom of choice to choose a total order on $I$, then $I$ can be taken to be any set. 
\end{remark}

\begin{defi}
Let $\{ G_i \}_{i \in I}$ be a collection of groups and $G= \bigoplus_{i \in I} G_i$. Define the support norm $\|.\|_{\supp}$ on $G$ to be the conjugation invariant norm counting the non-trivial summands in a direct sum decomposition of a given element in $G$. So for $g= \oplus_n g_n$ set
\[
\|g\|_{\supp} = \# \{ i : g_i \neq 1 \} . 
\]
So the norm $\|.\|_{\supp}$ is just the sum of discrete norms on the individual $G_i$, where the discrete norm is the norm that equals zero on the identity and one on every other element. 
\end{defi}

\begin{lemma}\label{equivsupp}
Let $(G_i,\|.\|_i)$ be a collection of groups together with conjugation invariant norms of uniformly bounded diameter and fineness. That is 
\[
\sup_i \sup_{g_i \in G_i} \|g_i\|_i < \infty \hspace{5mm} \text{and} \hspace{5mm} \inf_i \inf_{g_i \in G_i^\circ} \|g_i\|_i >0 ,
\]
where $G_i^\circ$ is the set of non-identity elements in $G_i$.
Then on $G= \bigoplus_{i \in I} G_i$ the norm $\|.\|_{supp}$ is equivalent to the norm $\|.\|$ given by the sum of the individual norms $\|.\|= \Sigma_{i \in I} \|.\|_i $. 
\end{lemma}

\begin{proof}
Let $g$ be given with $\|g\|_{supp}=k$. Then $g= g_{i_1}+ \dots + g_{i_k}$ where $g_{i_j} \in G_{i_j}$ are all non-identity elements. Then 
\begin{align*}
\|g \| = \|g_{i_1} \|_{i_1} + \dots + \| g_{i_k} \|_{i_k} & \leq k \sup_i \sup_{g_i \in G_i} \|g_i\|_i \\
\|g \| = \|g_{i_1} \|_{i_1} + \dots + \| g_{i_k} \|_{i_k} & \geq k \inf_i \inf_{g_i \neq 1} \|g_i\|_i ,
\end{align*}
which shows the equivalence.
\end{proof}

\begin{remark}\label{remindivnorm}
Note that the above applies to the case where all individual norms $\|.\|_i$ are word norms with uniformly bounded diameter. Then $\|.\|= \Sigma_i \|.\|_i $ is just the word norm generated by the union of all individual generating sets. 
\end{remark}

\begin{theorem}\label{suppcontractible}
Let $I$ be a totally ordered set. Let $\{G_i \}_{i \in I}$ be any collection of groups and let $G= \bigoplus_i G_i$. Then all asymptotic cones $\Cone_\omega(G,s_n)$ for $(G,\|.\|_{supp})$ and $m \geq 1$ are contractible. 
\end{theorem}

\begin{proof}
The support norm is just the sum $\Sigma_i \|.\|_i$ where all $\|.\|_i$ are the discrete norms on $G_i$, that is $\|g_i \|_i = 1$ if and only if $g_i \neq 1$ in $G_i$. Then the projections $p_i(g_i)=1 $ for all $g_i \in G_i$ satisfy all conditions of Proposition \ref{directsumcontraction} and its application implies the statement.  
\end{proof}

\begin{cor}
Let $F$ be a finite group and $\|.\|_F$ any norm on $F$. For any totally ordered set $I$ all asymptotic cones of the $I$-fold direct sum $G= \bigoplus_{i \in I} F$ with the $\ell_1$-norm $\|.\|= \Sigma_i \|.\|_F$ are contractible.   
\end{cor}

\begin{proof}
Any norm in a finite group assumes its infimum and supremum on the set of non-trivial elements. Hence, it is equivalent to the support norm by Remark \ref{remindivnorm}. It follows from Proposition \ref{suppcontractible} that the infinite direct sum only possesses contractible asymptotic cones. 
\end{proof}

Let us discuss one more example of a direct sum in which the norms on the factors are unbounded but Proposition \ref{directsumcontraction} still applies.

\begin{cor}
Let $I$ be a totally ordered set. Let $G= \bigoplus_{i \in I} \Z$ be equipped with the word norm $\|.\|$ coming from the standard basis $\{e_i \}_i$ and their inverses. Then all asymptotic cones of $(G,\|.\|)$ are contractible. 
\end{cor}

\begin{proof}
On every copy of the integers we have the projection map 
\begin{align*}
p \colon \Z \rightarrow \Z, \hspace{5mm} p(n)= 
\begin{cases}
n-1 & \text{if} \hspace{2mm} n \geq 1, \\
0  & \text{if} \hspace{2mm} n = 0, \\
n+1 & \text{if} \hspace{2mm} n \leq -1. \\
\end{cases}
\end{align*}
This map clearly fixes the identity element. Moreover, $p$ is distance decreasing for the standard word metric generated by $\{ \pm 1 \}$, satisfies $d(p(n),n) \leq 1$ for all $n \in \Z$ and reduces the norm of any non-trivial element. Hence, it follows from Proposition \ref{directsumcontraction}, that any asymptotic cone of $\bigoplus_{i \in I} \Z$ is contractible with respect to the word norm given by the sum of all generators of the individual summands. 
\end{proof}

\subsection{Infinite Symmetric Group}

In the spirit of the projection maps before we now define cutting maps $c_k \colon \Sym_\infty \rightarrow \Sym_\infty$ that cut off the last $k$ numbers of the support of any given permutation. Then natural numbers $n$ that are sent by $\sigma$ to the last $k$ numbers of the support of $\sigma$ will in $c_k(\sigma)$ be sent to their first return point not belonging to the last $k$-numbers of the support.  

More precisely, let $\sigma \in \Sym_\infty$ where $\supp(\sigma)= \{ i_1, \dots , i_\ell \}$ with $i_1 < \dots < i_\ell$. Define $c_k \colon \Sym_\infty  \rightarrow \Sym_\infty$ by \\
\begin{align*}
 \sigma \rightarrow c_k(\sigma)=
\begin{cases} 
c_k(\sigma)(i)=i & \text{if} \hspace{1mm} i > i_{\ell-k}, \\     
c_k(\sigma)(i)= \sigma(i)  & \text{if} \hspace{1mm} i \leq i_{\ell-k} \hspace{1mm} \text{and} \hspace{1mm} \sigma(i) \leq i_{\ell-k}, \\
c_k(\sigma)(i)= \sigma^m(i) &  \text{if} \hspace{1mm} i \leq i_{\ell-k} \hspace{1mm} \text{and} \hspace{1mm} \sigma(i) > i_{\ell-k}, \\
 \text{where} \hspace{1mm} m \hspace{1mm}\text{minimal such that} \hspace{1mm} \sigma^m(i) \leq i_{\ell-k}. &
\end{cases}
\end{align*}
Note, that $c_k(\sigma)$ is indeed a permutation satisfying $\supp(c_k(\sigma)) \subset \{i_1, \dots , i_{\ell-k} \}$ by definition and for $k \geq \ell$ we have $c_k (\sigma)=1$.

\begin{lemma} \label{cutcontinuous1}
The cutting maps satisfy $d(c_k(\sigma), c_m (\sigma)) \leq 2|k-m| $ for all $\sigma \in \Sym_\infty$ and all $k,m \in \N$, where $d$ denotes the induced metric by $\|.\|_{\supp}$.
\end{lemma}

\begin{proof}
Consider the case $m=k+1$. Writing $\sigma$ as a product of cycles we see that by definition $c_k(\sigma)$ and $c_{k+1}(\sigma)$ agree on all $i \in \{ i_1, \dots , i_{\ell-k} \}$ which are not $i_{\ell-{k}}$ or its preimage $\sigma^{-1}(i_{\ell-k})$. Hence, we have $d(c_k(\sigma), c_{k+1}(\sigma)) \leq 2$ and the lemma follows by induction. 
\end{proof}

\begin{lemma}\label{cutting}
Let $\sigma, \tau \in \Sym_\infty$ such that $\supp(\sigma) = \supp(\tau)$. Then $d(c_k(\sigma), c_k(\tau)) \leq d(\sigma, \tau)$ for all $k \in \N$. 
\end{lemma}

\begin{proof}
Without loss of generality assume that $\supp(\sigma)=\supp(\tau)= \{ 1, \dots , n \}$. Set 
\begin{align*}
Z= \{ z | \sigma(z) \neq \tau(z) \} &&
Z_k = \{ z | c_k(\sigma)(z) \neq c_k(\tau)(z) \}, 
\end{align*}
where by definition $ | Z| = \| \sigma \tau^{-1} \|_{supp}=d(\sigma, \tau)$ and 
\[ 
| Z_k| = \| c_k(\sigma) c_k(\tau)^{-1} \|_{\supp}= d(c_k(\sigma), c_k(\tau)).
\]
Our goal is to find a distinct corresponding element in $Z$ for every single element of $Z_k$, thus showing that $|Z_k| \leq |Z|$.

For $x> n-k$ we have $c_k(\tau)(x)= c_k(\sigma)(x)=x$ by definition of the cutting maps $c_k$. Let $x \leq n-k$ and distinguish the following cases 
\begin{enumerate}
\item $\sigma(x) \leq n-k$ and $\tau(x) \leq n-k$. Then $x \in Z_k$ if and only if $x \in Z$. 
\item $\sigma(x) \leq n-k$ and $\tau(x) > n-k$. Then $x \in Z$ and $x$ may or may not lie in $Z_k$. 
\item $\sigma(x) > n-k$ and $\tau(x) \leq n-k$. Then $x \in Z$ and $x$ may or may not lie in $Z_k$.  
\item $\sigma(x) > n-k$ and $\tau(x) > n-k$. Then write out both $\sigma$ and $\tau$ in their unique cycle decompositions. There exists a cycle in which $x=i_k= j_{k^\prime}$ is appearing of the form 
\begin{align*}
( i_1 \dots i_k i_{k+1} \dots i_\ell i_{\ell+1} \dots ) \subset \sigma, && (j_1 \dots j_{k^\prime} j_{k^\prime +1} \dots j_{\ell^\prime} j_{\ell^\prime +1} \dots) \subset \tau,
\end{align*}
such that $i_{\ell+1}$ and $j_{\ell^\prime +1}$ are the first natural numbers $\leq n-k$ which appear and all numbers in between are natural numbers larger than $n-k$. Then $c_k(\sigma)(x)=i_{\ell+1}$ and $c_k(\tau)(x)= j_{\ell^\prime +1}$. Now if $x \in Z_k$ then $c_k(\sigma)(x) \neq c_k(\tau)(x)$ and this means that there exists a first letter in $\{i_k, \dots , i_\ell \}$ and $\{j_{k^\prime}, \dots , j_{\ell^\prime} \}$ on which $\sigma$ and $\tau$ disagree. This is either $x$ itself or a letter $ > n-k$. In both cases this yields a unique $x^\prime \in Z$ which we have not counted in any of the other cases yet. 
\end{enumerate}
This shows every $x \in Z_k$ corresponds to some distinct $x^\prime \in Z$ which implies that $|Z_k | \leq |Z|$. Hence, statement of the lemma follows.  
\end{proof}

\begin{lemma}\label{cutcontinous2}
The maps satisfy $d(c_k(\sigma), c_k (\tau)) \leq 2d(\sigma, \tau) $ for all $\sigma, \tau \in \Sym_\infty$ and all $k \in \N$, where $d$ denotes the induced metric by $\|.\|_{\supp}$.
\end{lemma} 

\begin{proof}
The supports of $\sigma$ and $\tau$ can be written uniquely as disjoint unions
\begin{align*}
\supp(\sigma) = A \cup (\supp(\sigma) \cap \supp(\tau))  && \supp(\tau)= B \cup (\supp(\sigma) \cap \supp(\tau)),
\end{align*}
As in Lemma \ref{cutting} before use the notation 
\begin{align*}
Z= \{ z | \sigma(z) \neq \tau(z) \} &&
Z_k = \{ z | c_k(\sigma)(z) \neq c_k(\tau)(z) \}. 
\end{align*}
Clearly, we have $|Z| \geq \max \{|A|,|B| \}$. 

Let $k \in \N$ and consider the cutting map $c_k$. By definition of $c_k$ we erase $k$ numbers from both $\supp(\sigma)$ and $\supp(\tau)$ separately adhering to the total order on the natural numbers. Let $a \in A$ be the maximal element among all numbers in $\supp(\sigma) \cup \supp(\tau)$, then clearly $d(c_1(\sigma), \tau) \leq d(\sigma, \tau)-1$, since we just erased an element on which they were not the same. Consequently, in this case we have $d(c_1(\sigma),c_1(\tau)) \leq d(\sigma, \tau)+1$. When erasing a letter from $B$ we have the same behaviour. So whenever we erase a letter from $A$ or $B$ in the process of defining the cutting map we will increase distances by at most 1. 

However, we erase at least $\min \{ k-|A|, k-|B| \}$ common elements from $\supp(\sigma) \cap \supp(\tau)$ in both $\sigma$ and $\tau$, since there is a total order on these elements. For these the same argument of Lemma \ref{cutting} above shows that the elements of $Z_k \cap \supp(\sigma) \cap \supp(\tau)$ correspond to elements from $Z$. We deduce that $|Z_k| \leq |Z|+ \max \{|A|,|B| \}$. Since $\max \{|A|,|B| \} \leq |Z|$ it follows that $d(c_k(\sigma), c_k(\tau)) \leq 2d(\sigma, \tau)$ for all $k$. 
\end{proof}

Recall, that the equivalence of the norms and the quasi-isometry for $\Sym_\infty$ and $\A_\infty$ are discussed in Section 2.2. Moreover, coarsely equivalent norms and quasi-isometric spaces yield Lipschitz-isomorphic asymptotic cones by Proposition \ref{qi-invariant} and Proposition \ref{coarseequiv}. 

\begin{T} \label{Symcontractible}
All asymptotic cones of $(\Sym_\infty, \|.\|_{\supp})$ are contractible.
\end{T}

\begin{proof}
It is clear by definition of $c_k \colon \Sym_\infty \rightarrow \Sym_\infty$ that for all $\sigma \in \Sym_\infty$ we have $c_k(\sigma)=1$ if $\| \sigma \|_{\supp} \leq k $ and that $\|c_k(\sigma) \|_{\supp} \leq \| \sigma \|_{\supp} - k$ whenever $\| \sigma \|_{\supp} \geq k$. This yields the third condition of the Contraction Lemma \ref{ContractionLemma}. The first two conditions are guaranteed by Lemma \ref{cutcontinuous1} and Lemma \ref{cutcontinous2}. So the Contraction Lemma  \ref{ContractionLemma} applies and the result follows. 
\end{proof}

\subsection{Free Products}

The aim of this section is to apply our contraction lemma to recover classically known results regarding the contractibility of asymptotic cones of free products with respect to word norms associated to a finite generating set that is not invariant under conjugation. We derive very visual contraction maps from projection maps onto prefixes of words in this case. Thus, Section 4.3 and Proposition \ref{circleretract} in the appendix constitute the only parts in this paper in which the norms in consideration are \underline{not} conjugation invariant. An extensive study of asymptotic cones of hyperbolic groups with respect to word norms given by a finite generating set can be found in  \cite{Behrstock}, \cite{KapDru} and \cite{Osin}. In fact, for a finitely presented group $G$, hyperbolicity is equivalent to the existence of such an asymptotic cone that is a tree by \cite[Thm. 11.170]{KapDru}. 

\begin{defi}
Let $G= \ast_{i \in I} G_i$ be a free product of a family of groups $\{G_i\}_{i \in I}$ for some indexing set $I$. For each $i$ the factor $G_i$ is a subgroup of $G$ via the canonical inclusion. An element of $G$ that belongs to one of the factors is called a \textit{letter} of $G$. Any product of letters is called a \textit{word} in $G$. The product of any two letters belonging to the same factor in $G$ can be replaced by the letter that represents their product in that factor. Moreover, any identity letters appearing in a word can be omitted without changing the element the word represents in $G$.  Recall that any element $g \in G$ has a unique presentation as a word, where no two consecutive letters lie in the same factor and no identity letters appear. Such a word is called \textit{reduced}.
\end{defi}

\begin{defi}
Let $\{ (G_i,\|.\|_i) \}$ be a collection of groups together with norms for some indexing set $I$. Define a norm $\|.\|_{\ell_1}$ on the free product $G=\ast_{i \in I} G_i$ as follows. Write $g \in G$ uniquely as a reduced word $g=g_1 \dots g_k$ for some $k \in \N$ where $g_j \in G_{i_j}$ for all $j$. Then set \[
\|g\|_{\ell_1}= \|g_1\|_{i_1} + \dots + \|g_k\|_{i_k}.
\]
This defines a norm on the free product $G$ which we will refer to as the $\ell_1$\textit{-norm} \textit{associated to the collection} $\{ (G_i,\|.\|_i) \}$. 
\end{defi}

It follows from the definition of the $\ell_1$\textit{-norm} associated to a collection $\{ (G_i,\|.\|_i) \}$ that the projection onto the factor $G_j$ is Lipschitz-continuous with constant one with respect to the norms $\|.\|_{\ell_1}$ and $\|.\|_j$. 
 
\begin{lemma}\label{inclfreeprod}
Let $\{ (G_i,\|.\|_i) \}$ be a collection of groups together with norms for some set $I$. Then for any $j \in I$ the inclusion $\varphi_j \colon (G_j, \|.\|_i) \hookrightarrow (\ast_{i \in I} G_i, \|.\|_{\ell_1})$ is an isometry. 
\end{lemma}

\begin{proof}
The inclusion $\varphi_j$ maps any letter $g \in G_j$ to itself in $\ast_{i \in I} G_i$ and this is its unique presentation as a reduced word in $\ast_{i \in I} G_i$. Hence, we have $\| \varphi_j(g) \|_{\ell_1} = \|g \|_{\ell_1} = \|g \|_{j}$. 
\end{proof}

Recall, that the discrete norm on a group is the norm that takes value 0 on the identity element and value 1 on all other all other elements in the group. Then the associated $\ell_1$-norm is denoted by $\|.\|_{\rm supp}$ and measures the length of a reduced expression in terms of letters. The norm $\|.\|_{\rm supp}$ may be viewed as an analogue of the support norm before. Similar to the case of direct sums we have the following lemma. 

\begin{lemma}\label{freeproductequivalentnorms}
Let $\{(G_i, \|.\|_i) \}$ be any collection of groups $G_i$ together with norms $\|.\|_i$ of bounded diameter and bounded fineness. That is,  
\[
\sup_i \sup_{g_i \in G_i} \|g_i\|_i < \infty \hspace{5mm} \text{and} \hspace{5mm} \inf_i \inf_{g_i \in G_i^\circ} \|g_i\|_i >0 ,
\]
where $G_i^\circ$ is the set of non-identity elements in $G_i$.
Then the $\ell_1$-norm $\|.\|_{\ell_1}$ associated to $\{(G_i, \|.\|_i) \}$ is equivalent on $\ast_{i \in I} G_i$ to the $\ell_1$-norm $\|.\|_{\rm supp}$ associated to $\{(G_i, \|.\|_{d_i}) \}$, where $d_i$ is the discrete norm on $G_i$ for all $i$. 
\end{lemma}

\begin{proof}
Let $g=g_1 \dots g_k$ as a reduced word in $\ast_{i \in I} G_i$ with $g_j \in G_{i_j}$. Then $\|g\|_{\rm supp}=k$ and
\begin{align*}
\|g \|_{\ell_1} = \|g_1\|_{i_1} + \dots \|g_k \|_{i_k} \leq \sup_{h_1 \in G_{i_1}} \|h_1 \|_{i_1} + \dots + \sup_{h_k \in G_{i_k}} \|h_k \|_{i_k} \leq k \sup_i \sup_{h_i \in G_i} \|h_i\|_i, \\
 \|g \|_{\ell_1} = \|g_1\|_{i_1} + \dots \|g_k \|_{i_k} \geq \inf_{h_1 \in G_{i_1}^\circ} \|h_1 \|_{i_1} + \dots + \inf_{h_k \in G_{i_k}^\circ} \|h_k \|_{i_k} \geq k \inf_i \inf_{h_i \in G_i^\circ} \|h_i\|_i.
\end{align*}
Consequently, $\|.\|_{\ell_1}$ and $\|.\|_{\rm supp}$ are equivalent.
\end{proof}

\begin{theorem} \label{freeproductcontractible}
Let $I$ be a set. Let $\{ (G_i, \|.\|_i, p_i) \}_{i \in I}$ be a collection of groups $G_i$ together with integer valued norms $\|.\|_i$ and maps $p_i \colon G_i \rightarrow G_i$ which satisfy for all $i \in I$
\begin{enumerate}[label=(\roman*)]
\item $d_i(p_i(g),p_i(h)) \leq  d_i(g,h)$ for all $g,h \in G_i$,
\item $d_i(p_i(g),g) \leq 1 $ for all $g \in G_i$, 
\item $\| p_i(g) \|_i < \|g\|_i$ for all non-trivial $g \in G_i$,
\item $p_i(1)=1$, where $1 \in G_i$ is the trivial element,
\end{enumerate}
where $d_i$ denotes the metric induced by $\|.\|_i$. 
Then all asymptotic cones of \hspace{1mm}$G= \ast_{i \in I} G_i $ with respect to the $\ell_1$-norm associated to the collection $\{ (G_i, \|.\|_i, p_i) \}_{i \in I}$ are contractible.
\end{theorem}

\begin{proof}
Let $g \in G$ and let $g=g_1 \dots g_k$ be its unique presentation as a reduced word. Let $i_1$ be such that $g_1 \in G_{i_1}$. Define 
\begin{align*}
p \colon G \rightarrow G, \hspace{5mm} p(g)=p_{i_1}(g_{1})g_2 \dots g_k
\end{align*}
with $p(1)=1$, where $1$ is the identity element of $G$. Note, that $p_{i_1}(g_{i_1})g_2 \dots g_k$ is a reduced word after potentially omitting $p_{i_1}(g_1)$ if it is equal to the identity in $G_{i_1}$. For all non-trivial $g \in G$ it holds that
\begin{align*}
\|p(g)\|_{\ell_1}& = \|p(g_1)\|_{i_1} + \|g_2 \|_{i_2} + \dots \|g_k\|_{i_k} < \| g_1 \|_{i_1} + \|g_2 \|_{i_2} + \dots \|g_k\|_{i_k} = \|g \|_{\ell_1}
\end{align*}
and
\begin{align*}
d(p(g),g) &= \| p(g) g^{-1} \|_{\ell_1} = \| p_{i_1}(g_{1})g_2 \dots g_k (g_1 \dots g_k)^{-1} \|_{\ell_1} = \|p_{i_1} (g_1) g_1^{-1} \|_{\ell_1} \\
&= \|p_{i_1} (g_1) g_1^{-1} \|_{i_1}  = d_{i_1}(p_{i_1}(g_1),g_1) \leq 1.
\end{align*}

It remains to check that $p$ is distance decreasing since then by Lemma \ref{SingleprojectionContractionLemma} the Contraction Lemma \ref{ContractionLemma} applies and the result follows. Let $g,h \in G$ with presentations as reduced words given by $g= g_1 \dots g_k$ and $h=h_1 \dots h_\ell$. The only operations to obtain the reduced presentation of $gh^{-1}$ from the product of their respective reduced expressions are multiplying adjacent elements in the same group and omitting identities. So, if $g_k$ and $h_\ell$ are letters from different groups then $g_1 \dots g_k h_\ell^{-1} \dots h_1^{-1}$ is already reduced. If $g_k h_\ell^{-1}$ is not the identity, then there are no cancellations in $g_1 \dots g_k h_\ell^{-1} \dots h_1^{-1}$. After possibly performing a sequence of cancellations there is a positive integer $a \geq 0$ such that 
\begin{align*}
gh^{-1} = g_1 \dots g_k h_\ell^{-1} \dots h_1^{-1} = g_1 \dots g_{k-a} h_{\ell-a}^{-1} \dots h_1, 
\end{align*} 
where the latter is a reduced word. Clearly $a \leq \min \{k,\ell\}$. Distinguish the following cases. 

If $a < k-1$ then 
\begin{align*}
d(p(g),h)& = \|p(g)p(h)^{-1} \|_{\ell_1} \\ & = \| p_{i_1}(g_1) g_2 \dots g_{k-a} h_{\ell-a}^{-1} \dots h_1 \|_{\ell_1} \\
 & = \| p_{i_1}(g_1) \|_{i_1} + \|g_2 \dots g_{k-a} h_{\ell-a}^{-1} \dots h_1 \|_{\ell_1}\\
  & \leq \|g_1\|_{i_1}-1 + \|g_2 \dots g_{k-a} h_{\ell-a}^{-1} \dots h_1 \|_{\ell_1} \\
 & = \|g_1 g_2 \dots g_{k-a} h_{\ell-a}^{-1} \dots h_1 \|_{\ell_1} -1 \\
 & = d(g,h)-1, 
\end{align*}
where we used the fact that $p_{i_1}$ decreases the integer valued norm $\|.\|_{i_1}$. Consequently, \[
d(p(g),p(h)) \leq d(p(g),h)+d(h,p(h)) \leq d(g,h)-1 + 1 =d(g,h).
\]
Similarly, if $a < \ell-1 $ we obtain $d(p(g),p(h)) \leq d(g,h)$ with the same argument. 

Finally, if $ a \geq \max \{k-1,\ell-1 \}$ then $gh^{-1}= g_1  h_1^{-1}$. If $g_1$ and $h_1$ belong to different groups, then \[
d(p(g),p(h))= \|p(g_1) \|_{i_1} + \|p(h_1)^{-1} \|_{j_1} < \|g_1 \|_{i_1} + \|h_1^{-1} \|_{j_1} = \|gh^{-1} \|_{\ell_1} =d(g,h).
\]
If both of them belong to the same group, then \[
d(p(g),p(h))=d_{i_1}(p_{i_1}(g_1),p_{i_1}(h_1)) \leq d_{i_1}(g_1,h_1)=d(g,h).
\]
Thus $p$ is distance decreasing in all cases and the result follows from Lemma \ref{SingleprojectionContractionLemma}. 
\end{proof}

\begin{cor}
Let $\{(G_i, \|.\|_{d_i}) \}$ be any collection of groups $G_i$ all of which are equipped with the discrete norm $\|.\|_{d_i}$. Then all asymptotic cones of $\ast_{i \in I} G_i$ equipped with the associated $\ell_1$-norm are contractible.
\end{cor}

\begin{proof}
Take $p_i \colon G_i \rightarrow G_i$ to be the projection sending the whole group to the identity element. Then the collection $\{(G_i, \|.\|_{d_i},p_i) \}_{i \in I}$ trivially satisfies all the assumptions of Proposition \ref{freeproductcontractible}.  
\end{proof}

The equivalence of norms in Lemma \ref{freeproductequivalentnorms} yields the strengthening: 

\begin{cor} \label{freeproductcorollary}
Let $\{(G_i, \|.\|_i) \}$ be any collection of groups $G_i$ together with norms $\|.\|_i$ of bounded diameter and bounded fineness. That is  
\[
\sup_i \sup_{g_i \in G_i} \|g_i\|_i < \infty \hspace{5mm} \text{and} \hspace{5mm} \inf_i \inf_{g_i \in G_i^\circ} \|g_i\|_i >0 ,
\]
where $G_i^\circ$ is the set of non-identity elements in $G_i$. Then all asymptotic cones of $\ast_{i \in I} G_i$ equipped with the associated $\ell_1$-norm are contractible. \qed
\end{cor}

\begin{example}
One application of Proposition \ref{freeproductcontractible} is that any free product of integers with the standard metric only possess contractible cones. Namely, the projection map 
\begin{align*}
p \colon \Z \rightarrow \Z, \hspace{5mm} p(n)= 
\begin{cases}
n-1 & \text{if} \hspace{2mm} n \geq 1 \\
0  & \text{if} \hspace{2mm} n = 0 \\
n+1 & \text{if} \hspace{2mm} n \leq -1 \\
\end{cases}
\end{align*}
satisfies all the assumptions of Proposition \ref{freeproductcontractible}.  
\end{example}

\begin{example}
Let $\{(F_i, \|.\|_i) \}$ be any collection of finite groups together with arbitrary integer valued norms $\|.\|_i$ on them. Then by Corollary \ref{freeproductcorollary} all asymptotic cones of $(\ast_{i \in I} F_i)$ equipped with the associated $\ell_1$-norm are contractible.
\end{example}

\begin{example}
Let $\{(F_i, \|.\|_i) \}$ be any finite collection of finite groups together with arbitrary norms $\|.\|_i$ on them. Then the associated $\ell_1$-norm is equivalent to any word norm generated by finitely many elements in $\ast_{i \in I} F_i$. So by Corollary \ref{freeproductcorollary} all asymptotic cones of $\ast_{i \in I} F_i$ with respect to any such norm are contractible.
\end{example}

\section{Algebraic properties of the asymptotic cone of the infinite symmetric group}

Let us make a remark on the subgroups of $\Cone_\omega(\Sym_\infty, \|.\|_{\rm supp})$. First, any finite group $F$ is a subgroup of some finite symmetric group $\Sym_k$. Using disjoint copies of $\Sym_k$ in $\Sym_\infty$ one constructs for each $g \in F$ the sequence of elements $g_n = (g, \dots, g) \in F^{s_n} \leq (\Sym_k)^{s_n} \leq \Sym_\infty$. The sequence $(g_n)_n$ has support norm growing asymptotically like $s_n$ and this construction realises $F$ as a subgroup of $\Cone_\omega(\Sym_\infty, \|.\|_{\rm supp})$. Moreover, let $h_n=(1 \enspace 2 \enspace ... \enspace s_n)$ be the standard $s_n$-cycle. Then $h =[(h_n)_{n \in \N}]$ has infinite order in $\Cone_\omega(\Sym_\infty, \|.\|_{\rm supp})$. Disjointly arranging such $s_n$-cycles one can realise any finite free abelian group as a subgroup of $\Cone_\omega(\Sym_\infty, \|.\|_{\rm supp})$. Additionally, $\Cone_\omega(\Sym_\infty, s_n, \|.\|_{\rm supp})$ contains the group $\mathfrak{S}(s_n/\omega)= \Cone_\omega(\Sym_{s_n}, s_n, \|.\|_{\rm supp})$ that is universally sophic according to \cite[Thm. 1]{Elek}.

The main object of this section is to prove the following theorem. Note that the simplicity statement is similar to the one for universally sophic groups given in \cite{Elek}.  

\begin{theorem}\label{symmetricgroupthm}
Up to bi-Lipschitz isomorphism $\Sym_\infty$ and $\A_\infty$ define the same asymptotic cone $\Cone_\omega(\Sym_\infty, \|.\|_{\rm supp})$ with respect to the support norm or any word norm generated by finitely many conjugacy classes. For all scaling sequences and ultrafilters $\Cone_\omega(\Sym_\infty, \|.\|_{\rm supp})$ is contractible as a topological space and as a group $\Cone_\omega(\Sym_\infty, \|.\|_{\rm supp})$ is simple and uniformly perfect of commutator length one .
\end{theorem}

\begin{proof}
Since $\A_\infty \hookrightarrow \Sym_\infty$ is a quasi-isometry and by Lemma \ref{symequiva} and Lemma \ref{alternatingequiva} the support norms are equivalent on $\A_\infty$ and  $\Sym_\infty$ to any word norm generated by finitely many conjugacy classes all their asymptotic cones are bi-Lipschitz isomorphic as metric groups. The results on algebraic properties are in Proposition \ref{Sinftyperfect} and Proposition \ref{Sinftysimple} whereas contractibility was shown in Theorem \ref{Symcontractible}. 
\end{proof}

\subsection{Perfectness of the Cone}

\begin{theorem}\label{Sinftyperfect}
$\Cone_\omega(\A_\infty, \|.\|_{\rm supp})$ is a perfect group of commutator length one. 
\end{theorem}

\begin{proof}
Let $a=[(a_n)_{n \in \mathbb{N}} ] \in \Cone_\omega(\A_\infty, \|.\|_{\rm supp})$ be an arbitrary element. For every $n$ there exists $k(n)$ such that $a_n \in A_{k(n)}$. In fact, $a_n$ can be viewed as an element of $A_{k}$ where $k= \max \lbrace \|a_n\|_{\rm supp}, 5 \rbrace$ since $a_n$ only permutes the elements within its support. Since every element in $A_k$ is a commutator for $k \geq 5$ (originally due to \cite[Thm. 1]{Miller}, but reproved in \cite [Thm. 7]{Ore}),  there exist $b_n, c_n \in A_k$ such that $[b_n,c_n]=a_n$. Then $\|b_n\|_{\rm supp}$ and $\|c_n\|_{\rm supp}$ are bounded by $\|a_n\|_{\rm supp}+ 5$. Hence, the sequences $(b_n)$,$(c_n)$ are are admissible and represent elements $b$,$c$ in the asymptotic cone $\Cone_\omega(\A_\infty, \|.\|_{\rm supp})$. Then $a=[b,c]$ in $\Cone_\omega(\A_\infty, \|.\|_{\rm supp})$, which in turn also implies non-triviality of $b$ and $c$ whenever $a$ is non-trivial.
\end{proof}

Since every element in $\Cone_\omega(\A_\infty, \|.\|_{\rm supp})$ is an elementary commutator, there is no non-trivial homogeneous quasimorphism on $\Cone_\omega(\A_\infty, \|.\|_{\rm supp})$.

\subsection{Simplicity}

\begin{lemma}\label{4}
Let $G$ be a group. If for all non-trivial $g \in G$ there exists a subgroup $H < G$ such that $[H,H]$ normally generates $G$ and $H$ commutes with $gHg^{-1}$, then $G$ is a simple group. 
\end{lemma}

\begin{proof}
Assume $N$ is a non-trivial normal subgroup of $G$ and let $g \neq 1$ be an element of $N$. Given any elementary commutator $[a,b] \in [H,H]$ it holds that 
\begin{align*}
[[a,g],b]=aga^{-1}g^{-1} b gag^{-1}a^{-1} b^{-1} = a ga^{-1}g^{-1}gag^{-1} b a^{-1} b^{-1} =[a,b] .
\end{align*}
That is, every elementary commutator is a product of conjugates of $g$. Therefore, $[H,H] \leq N$ and since $[H,H]$ normally generates the whole group, so does $N$. Consequently, $G$ does not admit any proper normal subgroups apart from the trivial one.
\end{proof}

\begin{lemma}\label{1}
Any $\sigma \in \Sym_\infty$ with $\|\sigma\|_{\rm supp}=n$ displaces a set of size at least $\frac{n}{3}$ from itself. 
\end{lemma}

\begin{proof}
First, write $\sigma$ uniquely as a product of non-intersecting cycles $\sigma=\tau_1 \dots \tau_\ell$. Any cycle of size $k$ is conjugate to the $k$-cycle $(1 \hspace{1mm} ... \hspace{1mm} k)$. If $k$ is even, the set containing all odd natural numbers smaller than $k$ is displaced from itself. If $k$ is odd, then the set of all odd numbers strictly smaller than $k$ is displaced from itself. Hence, for any cycle of length $k$ we find a set of cardinality $\geq k/3$ which is displaced from itself and the same holds for $\sigma$ since all cycles $\tau_i$ have disjoint support. 
\end{proof}

\begin{lemma}\label{3}
Let $\sigma \in \Sym_\infty$ with $\|\sigma\|_{\rm supp}=n$. Then for any $k \in \mathbb{N}$ there exist $\sigma_1$,$\sigma_2$ with $\|\sigma_1\|_{\rm supp} \leq k$, $\|\sigma_2\|_{\rm supp} \leq n-k+1$ and $\sigma=\sigma_1 \sigma_2$. 
\end{lemma}

\begin{proof}
Again, use the unique disjoint cycle decompostion of $\sigma$ to write $\sigma=\tau_1 \cdot \tau_\ell$ for some $\ell$. Consider the $k$-th number appearing in this cycle decomposition. If this happens to be the last element of a cycle, then the statement is immediate. Otherwise consider that cycle $\tau_i$ containing the $k$-th number appearing in the cycle decomposition. Without loss of generality we can assume $ \tau_i=(1 ... j)$. It can be split up at any $m \leq j$ to the two cycles 
\begin{align*}
(1 ... j) = (m ... j)(1 ... m) 
\end{align*}
and we have that their supports add up to $m+1$. The result follows.  
\end{proof} 

\begin{theorem}\label{Sinftysimple}
All asymptotic cones of $(\Sym_\infty,\|.\|_{\rm supp})$ are simple groups. 
\end{theorem}

\begin{proof}
Let $g \in \Cone_\omega(\Sym_\infty, \|.\|_{\rm supp})$ be non-trivial. By Lemma \ref{2} there exists a representative sequence $(g_n)$ of $g$ with $\|g_n\|_{\rm supp} \geq Ks_n$ for some $K > 0$. By Lemma \ref{1} every $g_n$ displaces some set $D_n$ of size at least $ \frac{K s_n n}{3}$ from itself. Define 
\begin{align*}
H:= \left \{ h=[(h_n)] \in \Cone_\omega(\Sym_\infty, \|.\|_{\rm supp}) \hspace{1mm} \mid \hspace{1mm} \supp(h_n) \subset D_n \right \} .
\end{align*}
Then $H$ and $gHg^{-1}$ commute by construction. 

It remains to show that the commutator $[H,H]$ normally generates $\Cone_\omega(\Sym_\infty, \|.\|_{\rm supp})$. Let $a=[(a_n)] \in \Cone_\omega(\Sym_\infty, \|.\|_{\rm supp})$ be arbitrary. By definition of the asymptotic cone there exists $C > 0 $ such that $\|a_n\| \leq Cs_n$ for all $n$. Apply Lemma \ref{3} to write $a_n = b_{1,n}c_n$ where the support of $b_{1,n}$ is of size at most $ |D_n| $ and the size of the support of $c_n$ is bounded by $Cn- |D_n|+1$. Proceed by applying Lemma \ref{3} again to chop up $c_n$ further. Inductively, after $m \leq \frac{Cs_n}{Ks_n/3}= \frac{3C}{K}$ iterations we have factored $a_n=b_{1,n} ... b_{m,n} b^\prime_n $ where the supports of each $b_{i_n}$ is smaller than $|D_n|$ and the support of $b^\prime_n$ is at most of size of the number of iterations of our chopping process which is bounded by $\frac{3C}{K}$. For $n$ large enough $\frac{3C}{K} \leq |D_n|$ and since changing finitely many elements of a sequence does not change the ultralimit we just assume this is true for all $n$ from here onwards. Set $b^\prime_n=b_{m+1,n}$. Then for each $b_{i,n}$ there exists a $\gamma_{i,n} \in \Sym_\infty$ which swaps the support of $b_{i,n}$ with a subset of $D_n$. Moreover, $\supp(\gamma_{i,n}) \subset \supp(b_{i,n}) \cup D_n$ and so $(\gamma_{i,n})_{n \in \N}$ is an admissable sequence for all $i$. 

Finally, $\gamma_{i,n} b_{i,n} \gamma_{i,n}^{-1}$ is supported in $D_n$. Without loss of generality we may assume that  $\gamma_{i,n} b_{i,n} \gamma_{i,n}^{-1}$ is an even permutation  since we can multiply it by some transposition supported in $D_n$ without changing the ultralimit. By \cite[Thm. 1]{Miller} (or \cite[Thm. 7]{Ore}) $\gamma_{i,n} b_{i,n} \gamma_{i,n}^{-1}$ can be written as an elementary commutator $[c_{i,n},d_{i,n}]$ of elements $c_{i_n}, d_{i,n}$ supported in $D_n$. The sequences $(b_{i,n})_n$, $(\gamma_{i,n})_n$, $(c_{i,n})_n$ and $(d_{i,n})_n$ are all admissible and represent elements $b_i,\gamma_i,c_i,d_i$ in $\Cone_\omega(\Sym_\infty, \|.\|_{\rm supp})$ for all $i$. Then
\begin{align*}
g= b_1...b_{m+1} = \gamma_1^{-1} [c_1,d_1] \gamma_1 ... \gamma_{m+1}^{-1} [c_{m+1},d_{m+1}] \gamma_{m+1} \hspace{1mm}.
\end{align*}
Thus, $\Cone_\omega(\Sym_\infty, \|.\|_{\rm supp})$ is simple by Lemma \ref{4}. 
\end{proof}

\begin{remark}
Let $g \in (\Sym_\infty,\|.\|_{\rm supp})$ Then the above proof shows that the number of factors required to express $h \in (\Sym_\infty,\|.\|_{\rm supp})$ as a product of conjugates of $g$ grows linearly in ${\|h \|_{\supp}}$.
\end{remark} 

\begin{remark}
An alternative strategy to prove Proposition \ref{Sinftysimple} relying on the simplicity shown in \cite[Thm 1]{Elek} is the following. Given two elements $g,h \in \Cone_\omega(\Sym_\infty,\|.\|_{\rm supp}, s_n)$ find a suitable common subgroup isomorphic to $\mathfrak{S}(s^\prime_n/\omega)$ where $s^\prime_n = k s_n$ for some fixed $k \in \N$ and all $n$. Since $\mathfrak{S}(s^\prime_n/\omega)$ is simple by \cite[Thm 1]{Elek} $h$ can be expressed as a product of conjugates of $g$ and its inverse in $\mathfrak{S}(s^\prime_n/\omega)$. Then, the same follows in $\Cone_\omega(\Sym_\infty,\|.\|_{\rm supp}, s_n)$.
\end{remark}

\subsection{Coarse equivalence of support norm on the cone}

By abuse of notation we denote the norm induced by $\|.\|_{\supp}$ on $\Cone_\omega(\Sym_\infty, \|.\|_{\rm supp})$ by $\|.\|_{\supp}$ again. The norm $\|.\|_{\supp} \colon \Cone_\omega(\Sym_\infty, \|.\|_{\rm supp}) \to \R$ is continuous and surjective on the positive real axis and thus not equivalent to any norm generated by finitely many conjugacy classes. Moreover, the norm $\|.\|_{\supp}$ is stably bounded on $\Cone_\omega(\Sym_\infty, \|.\|_{\rm supp})$. This section shows that $\|.\|_{\supp}$ is coarsely equivalent to word norms generated by finitely many conjugacy classes in a very explicit way and concludes that every conjugation invariant norm is stably bounded on $\Cone_\omega(\Sym_\infty, \|.\|_{\rm supp})$ despite the fact that it contains many elements of infinite order. 

\begin{T}\cite[Theorem 3.05]{Brenner} \label{effectiveconjugation}
Let $\sigma \in \A_n$. Suppose that $\sigma$ has an orbit of length two and $n -2r \geq -1$, where $r$ is the number of orbits of $\sigma$. Then $C_\sigma^4=A_n$, where $C_\sigma$ is the conjugacy class of $\sigma$. \qed
\end{T}

Note, that $\sigma \in \A_n$ with $\| \sigma \|_{\supp}=n$ has at most $n/2$ orbits.  If $\sigma$ has an orbit of length two then Theorem \ref{effectiveconjugation} implies that $C_\sigma^4=A_n$.

\begin{theorem}
On $\Cone_\omega(\Sym_\infty, \|.\|_{\rm supp})$ the induced norm $\|.\|_{\supp}$ is coarsely equivalent to every word norm generated by finitely many conjugacy classes. 
\end{theorem}

\begin{proof}
By Lemma \ref{equiva} it suffices to show this for a single word norm generated by finitely many conjugacy classes. Let $g \in \Cone_\omega(\Sym_\infty, \|.\|_{\rm supp})$ be non-trivial. Since $\Cone_\omega(\Sym_\infty, \|.\|_{\rm supp})$ is simple, the conjugacy classes of $g$ and $g^{-1}$ define a conjugation invariant norm $\|.\|_g$ on $\Cone_\omega(\Sym_\infty, \|.\|_{\rm supp})$. 
By Lemma \ref{fcgmaximal} it suffices to show that there exist $A,B>0 \in \R$ such that $\|.\|_g \leq A \|.\|_{\supp}+B$ to conclude that $\|.\|_g$ and $\|.\|_{\supp}$ are coarsely equivalent.

By Lemma \ref{2} there exists a representative sequence $(g_n)_n$ of $g$ such that $\|g_n\|_{\rm supp} \geq Ks_n$ for some $K > 0$. Furthermore, modifying $g_n$ by at most two transpositions outside its support we may assume additionally that each of the $g_n$ has an orbit of length 2 and all $g_n$ are elements in $\A_\infty$. Since this modification still represents $g$ in $\Cone_\omega(\Sym_\infty, \|.\|_{\rm supp})$, assume $g_n$ to be element in $\A_\infty$ and contain an orbit of order two for all $n$. 

Let $h \in \Cone_\omega(\Sym_\infty, \|.\|_{\rm supp})$.  By Lemma \ref{effectiveupperbound} $h$ is represented by a sequence $(h_n)$ satisfying $\|h_n \|_{\supp} \leq 2\|h\|_{\supp} \cdot s_n$. As before in the proof of Proposition \ref{Sinftysimple} we now invoke Lemma \ref{3} in order to chop up $h_n$ in $m$ steps into pieces. That is, after $m$ chops we obtain a product decomposition 
\[
h_n= b_{1,n} ... b_{m,n} b^\prime_n
\]
where $\|b_{i,n}\|_{\supp} \leq Ks_n$ for all $i$ and $\|b^\prime_n \|_{\supp} \leq m$. The number of steps $m$ in the chopping procedure is bounded by  
\[
m \leq \frac{2\|h\|_{\supp} \cdot s_n}{ Ks_n } = \frac{2\|h\|_{\supp}}{K} .
\]
Thus $ \|b^\prime_n \|_{\supp} \leq \frac{2C}{K} \leq K s_n$ for large $n$ and so we may assume without loss of generality that this is the case for all $n$. Set $b_{m+1,n}=b^\prime_n$. For all $i,n$ there exists $c_{i,n}$ interchanging the support of $b_{i,n}$ with a subset of the first $Ks_n$ natural numbers which themselves have support bounded by $\| c_{i,n} \|_{\supp} \leq 2Ks_n$. Similarly, there exist $d_n$ interchanging the support of $g_n$ with the first $\|g_n \|_{\supp}$ natural numbers and having support norm $\|d_n \|_{\supp} \leq 2 \|g_n \|_{\supp}$. 

Summarising, for all $i,n$ the elements $c_{i,n} b_{i,n} c_{i,n}^{-1}$ are supported on a subset of the first $Ks_n$ natural numbers which is a subset of the support of $d_n g_n d_n^{-1}$. Thus, apply Proposition \ref{effectiveconjugation} and deduce that $c_{i,n} b_{i,n} c_{i,n}^{-1}$ is a product of at most four conjugates of $g_n$ which all have support norm bounded by the support of $g_n$. That is 
\[
c_{i,n} b_{i,n} c_{i,n}^{-1} = \gamma_{1,n} g_n \gamma_{1,n}^{-1} \gamma_{2,n} g_n \gamma_{2,n}^{-1} \gamma_{3,n} g_n \gamma_{3,n}^{-1} \gamma_{4,n} g_n \gamma_{4,n}^{-1}
\]
where $\| \gamma_{j,n} \|_{\supp} \leq \| g_n \|_{\supp}$ for all $j$. 
 Consequently, since all involved sequences of elements are admissible, they yield elements in the asymptotic cone $\Cone_\omega(\Sym_\infty, \|.\|_{\rm supp})$. Then $h$ in $\Cone_\omega(\Sym_\infty, \|.\|_{\rm supp})$ is the product
\[
h= \prod_{i=1}^{m+1} b_i =  \prod_{i=1}^{m+1} c_i^{-1} \gamma_{1} g \gamma_{1}^{-1} \gamma_{2} g \gamma_{2}^{-1} \gamma_{3} g \gamma_{3}^{-1} \gamma_{4} g \gamma_{4}^{-1} c_i ,
\]
which is a product of at most $4m+4 \leq \frac{8\|h\|_{\supp} }{K}+4$ conjugates of $g$. Since $h$ was arbitrary to begin with this shows that 
\begin{align*}
\|.\|_g \leq \frac{8}{K} \|.\|_{\supp} + 4 . 
\end{align*}
\end{proof}

\begin{cor}
$\Cone_\omega(\Sym_\infty, \|.\|_{\rm supp})$ does not admit a stably unbounded norm. \qed
\end{cor}

\section{More examples: some classical metric ultraproducts}

Let us recall the definition of a metric ultraproduct. For more background see \cite{Stolz}, \cite{Wilson} and \cite{Wilson2}. Given a sequence $(G_n,d_n)$ of groups equipped with bi-invariant metrics; that is metrics that are invariant under left and right multiplication in $G_n$. Assume that the diameters $\sup_n \{ d_n(1,g_n)  \mid  g_n \in G_n \}$ are bounded independently of $n$. Then define the metric ultraproduct to be the quotient of the cartesian product $\prod_n G_n$ by the subgroup of elements with vanishing ultralimit
\begin{align*}
\prod_{n \to \omega} (G_n,d_n) =  \prod_n G_n \Big / \{ (g_n)_{n \in \N} \hspace{1mm} \mid \hspace{1mm} \lim_\omega d_n(1,g_n) = 0 \}.
\end{align*}

It becomes a metric space by setting $d= \lim_\omega d_n$. It is immediate from the definition that for any scaling sequence $s_n$ setting $d^\prime_n = s_n \cdot d_n$ we have  
\begin{align*}
\Cone_\omega(G_n, d^\prime_n, s_n) = \prod_{n \to \omega} (G_n,d_n).
\end{align*}

\begin{example}
Consider $(\Sym_n, d_n)$, where $d_n$ is the so-called Hamming distance induced by $\frac{\|.\|_{\rm \supp}}{n}$. Then $\prod_{n \to \omega} (\Sym_n,d_n)= \Cone_\omega(\Sym_n, \|.\|_{supp})$ is the asymptotic cone with respect to the standard scaling sequence $s_n=n$. It is one of the universally sophic groups considered in \cite{Elek} and thus simple. The contraction map of $\Cone_\omega(\Sym_\infty, \|.\|_{supp})$ in Theorem \ref{Symcontractible} restricts to a contraction map of $\prod_{n \to \omega} (\Sym_n,d_n)$. This remains true for $\prod_{n \to \omega} (\Sym_{s_n}, d_{s_n})$ and any other scaling sequence $(s_n)$. 
\end{example}

\subsection{Contraction Lemma for sequences of embedded spaces}

For simplicity of notation we will assume $s_n=n$ from here onwards.

\begin{lemma}[Contraction Lemma for Sequences]\label{GeneralContractionLemma}
Let $(X_n, d_n)$ be a sequence of subspaces such that $X_k \subset X_n$ for all $ k \leq n$, $d_n$ restricts to $d_k$ on $X_k$ for all $k \leq n$. Assume that $X_0$ consists of a single point $x_0$. Finally, assume that the exist maps $p_n \colon X_n \rightarrow X_{n-1}$ satisfying 
\begin{enumerate}[label=(\roman*)]
\item there exits $K \in \N$ such that for all $n$ and all $x_n \in X_n$ we have $d_n(p_n(x_n),x_n) \leq K$, 
\item $d_{n-1}(p_n(x_n),p_n(y_n)) \leq d_n(x_n,y_n)$ for all $x_n,y_n \in X_n$ and $n \in \N$.
\end{enumerate}
Then $\Cone_\omega(X_n, d_n)$ is contractible.  
\end{lemma}

\begin{proof}
Define maps $p_{n,k} \colon X_n \rightarrow X_{n-k}$ by the concatenation $p_{n,k}= p_{n-k+1} \circ \dots \circ p_n$ where we use the convention that for $k > n$ the maps $p_{n,k}$ map every element in $X_n$ to the single point in $X_0$.
\begin{claim}
These maps satisfy the following properties
\begin{enumerate}[label=(\roman*)]
\item $d_{n}( p_{n,k}(x),p_{n,\ell}(x)) \leq K|k-\ell|$ for all natural numbers $n,k,\ell$ and all $x \in X_n$,  
\item $ d_{n-k}(p_{n,k}(x), p_{n,k}(y)) \leq d_n(x,y)$ for all $n,k \in \N$ and all $x,y \in X_n$. 
\end{enumerate} 
\end{claim}
For $\ell < k$ it holds that 
\begin{align*}
d_{n}(p_{n,k}(x), p_{n, \ell} (x))  & \leq d_n(p_{n,k}(x), p_{n,\ell+1}(x)) + d_n( p_{n, \ell+1}(x), p_{n, \ell} (x)) \\
& = d_n(p_{n,k}(x), p_{n,\ell+1}(x)) + d_n \big ( p_{n-\ell}(p_{n,\ell}(x)), p_{n, \ell} (x) \big )  \\
& \leq d_n(p_{n,k}(x), p_{n,\ell+1}(x)) + K ,
\end{align*}
and so $(i)$ follows by induction. The second part of the claim follows from assumption $(ii)$ by induction on $k$ fixing $n$ since all inclusions are isometries. This establishes the claim. 

Define the contracting homotopy $H \colon \Cone_\omega(X_n, d_{n}) \times [0,1] \rightarrow \Cone_\omega(X_n,d_{n})$ by
 \begin{align*} \big ( [x_n], t \big) \rightarrow [p_{n,\lambda_n}(x_n)], \hspace{1mm} \text{where} \hspace{1mm} (\lambda_n) \subset \N \hspace{1mm} \text{is such that}  
\hspace{1mm} \lim_\omega \frac{\lambda_n}{n}= t.
\end{align*}

First, let us verify that $H$ is well-defined. Let $[x_n]=[y_n]$. Then 
\begin{align*}
\lim_\omega \frac{1}{n} d_n( p_{n,\lambda_n}(x_n),  p_{n,\lambda_n} (y_n)) \leq \lim_\omega \frac{1}{n} d_n(x_n y_n) =0.
\end{align*}
Clearly, for any $t \in [0,1]$ a sequence $(\lambda_n) \subset \N $ such that $ \lim_\omega \frac{\lambda_n}{n}= t$ can be chosen. Moreover, for another sequence $(\tau_n)$ such that $\lim_\omega \frac{\tau_n}{n}= t = \frac{\lambda_n}{n}$ it holds that
\begin{align*}
\lim_\omega \frac{1}{n} d_n(p_{n,\lambda_n}(x_n), p_{n,\tau_n}(x_n)) \leq \lim_\omega \frac{K}{n} | \lambda_n - \tau_n | = K \left | \lim_\omega \frac{\lambda_n}{n} - \lim_\omega \frac{\tau_n}{n} \right |= K|t-t|=0,
\end{align*}
which shows independence of the particular choice of sequence $(\lambda_n)$. Thus, $H$ is well-defined.

\begin{claim}
$H$ is continuous.
\end{claim}
 By Lemma \ref{continuity on factors} $H$ is continuous if it is continuous when restricted to each factor and $H$ is globally Lipschitz continuous with respect to one of the factors. Let us start with the time parameter. Denote the induced metric on $\Cone_\omega(X_n,d_n)$ simply by $d$. Fix $x=[x_n]$. By property $(i)$ we calculate for $\lambda_n$ and $\tau_n$ such that $\lim_\omega \frac{\lambda_n}{n}=t$ and $\lim_\omega \frac{\tau_n}{n}= t^\prime$ that
\begin{align*}
d(H(x,t),H(x,t^\prime))=\lim_\omega \frac{d_n( p_{n,\lambda_n}(x_n), p_{n,\tau_n}(x_n))}{n} \leq \lim_\omega \frac{K|\lambda_n-\tau_n|}{n} = K|t-t^\prime|.
\end{align*}
For Lipschitz continuity with respect to a fixed time $t$ we calculate using $(ii)$
\begin{align*}
d(H(x,t),H(y,t))= \lim_\omega \frac{d_n(p_{n,\lambda_n}(x_n), p_{n,\lambda_n}(y_n))}{n} \leq \lim_\omega \frac{d_n(x_n, y_n)}{n}= d(x,y)
\end{align*}
This establishes continuity of $H$. 

Finally, $H(-,0)=id$ letting $\lambda_n = 0$ for all $n$. Respectively, letting $\lambda_n=n$ for all $n$. $H(-,1) \equiv x_0$, the unique point $x_0 \in X_0$. Therefore, $H$ is the desired contracting homotopy.

\end{proof}

\subsection{Upper triangular matrices}

\begin{defi}
Let $R$ be a commutative ring. Define the rank norm of an element $g$ in the infinite general linear group $\GL_\infty(R)= \colim_n \GL_n(R)$ to be
\[
\|g\|_{rk}= \rk(g-id).
\]
Moreover, for $g \in \GL_n(R)$ we define the projectivised rank norm to be $\|g\|_{\overline{\rk}}= \frac{\|g\|_{\rk}}{n}$.
\end{defi}

Observe, that the (projectivised) rank norms are conjugation invariant and that by definition $\|g\|_{\rk}= \dim(\im(g-id))=n- \dim \Fix(g)$ where $\Fix(g)$ denotes the subspace of fixed points of $g$.

\begin{theorem}
Let $\B_n(R)$ be the group of invertible upper triangular matrices with entries in a commutative ring $R$. Then the metric ultraproduct $\prod_{n \to \omega} (B_n,\|.\|_{\overline{\rk}})$ is contractible.  
\end{theorem}

\begin{proof}
Clearly, the diameter of $B_n$ with the projectivised rank norm is bounded by 1. The statement is equivalent to showing that the cone $\Cone_\omega(\B_n(R),\|.\|_{\rk})$ for the standard scaling sequence $s_n=n$ is contractible. For $n \in \N$ define the projection maps $p_n \colon \B_n(R) \rightarrow \B_{n-1}(R)$ to be the projections onto the upper left block with the convention that $\B_0(R)$ is the trivial group. These maps satisfy
\begin{enumerate}[label=(\roman*)]
\item $p_n$ is a group homomorphism for all $n$,
\item $\| p_n(g)g^{-1} \|_{\rk} \leq 1$ for all $n$,
\item $\| p_n(g) p_n(h)^{-1} \|_{\rk} \leq \|gh^{-1} \|_{\rk}$ for all $g,h$ and all $n$.  
\end{enumerate}
Indeed, it is clear that $p_n$ is a group homomorphism and that $\|p_n(g)g^{-1}\|_{|rk} \leq 1$ since $p_n(g)g^{-1}$ differs from the identity in at most one column. Moreover, for $g \in \B_n$ it holds that 
\[
\|p_n(g)\|_{\rk} = \rk(p_n(g)-id) \leq \rk(g-id) = \|g\|_{\rk},
\]
Then it follows for $h,g \in \B_n$ using the fact that $p_n$ is a group homomorphism that
\[
\|p_n(g)p_n(h)^{-1} \|_{\rk} = \|p_n(gh^{-1})\|_{\rk} \leq \|gh^{-1} \|_{\rk}.
\]
Consequently, the contractibility of $\Cone_\omega(B_n(R),\|.\|_{\rk})$ follows from Lemma \ref{GeneralContractionLemma}. 
\end{proof}

\subsection{Symmetric positive definite matrices}

Let $\Po_n$ be the set of real symmetric positive definite matrices equipped with the rank norm $\|.\|_{\rk}$ as a subspace $\Po_n \subset \GL_n(\R)$. For $n \in N$ define the projection map 
\begin{align*}
p_n \colon \Po_n \rightarrow \Po_{n-1}, \hspace{5mm} A \rightarrow A^{(n-1)}
\end{align*}
by forgetting the last row and column. The projection maps $p_n$ are well-defined since $A^{(n-1)}$ is clearly symmetric and by the criterion of positive principle minors $A^{(n-1)}$ is still positive definite hence invertible. 

\begin{theorem} \label{Possymcontractible}
The metric ultraproduct $\prod_{n \to \omega} (\Po_n,\|.\|_{\overline{\rk}})$ is contractible.
\end{theorem}

\begin{proof}
Again, we will verify the conditions of Lemma \ref{GeneralContractionLemma} to conclude the contractibility of $\Cone_\omega(\Po_n, \|.\|_{\rk}) = \prod_{n \to \omega} (\Po_n,\|.\|_{\overline{\rk}})$. Let $A \in \Po_n$. Then 
\[
\|p_n(A)A^{-1}\|_{\rk} = \|A^{(n-1)}A^{-1} \|_{\rk}=\rk (A^{(n-1)}-A) \leq 2
\]
since the only possible non-zero entries of $A^{(n-1)}-A$ are in the $n$-th row and column. Furthermore, $A,B \in \Po_n$ satisfy 
\begin{align*}
\|p_n(A)p_n(B)^{-1}\|_{\rk} = \rk(A^{(n-1)}- B^{(n-1)}) \leq \rk(A-B)= \|AB^{-1}\|_{\rk}
\end{align*}
since any collection of linearly independent vectors of $A^{(n-1)}- B^{(n-1)}$ stays linearly independent in $A-B$. Thus, Lemma \ref{GeneralContractionLemma} applies and it follows that $\Cone_\omega(\Po_n, \|.\|_{\rk})$ is contractible.
\end{proof}

\subsection{Special orthogonal group}

Call an element $g \in \SO(n)$ that belongs to a subgroup conjugate to $\SO(2) \leq SO(n)$ an \textit{elementary rotation}. That is, elementary rotations in $\SO(n)$ are precisely the elements that fix a subspace of codimension two.

\begin{lemma}
The inclusions $\iota_n \colon \SO(n) \rightarrow \Or(n)$ are all isometries which are quasi-surjective with constant 1 with respect to the rank norm.  
\end{lemma}

\begin{proof}
Since all spaces are equipped with the same subspace norm of $\GL_n(\R)$ all inclusions are isometries by definition. For any $g \in \Or(n)$ we have $|\det(g)|=1$ and so the matrix $a=(\det(g))$ lies in $ \Or(1)$. Hence, $h= ag$ satisfies $\|hg^{-1} \|_{\rk} \leq 1$.  
\end{proof}

\begin{cor}
$\Cone_\omega(\SO(n),\|.\|_{\rk})= \Cone_\omega(\Or(n), \|.\|_{\rk})$. That is, the metric ultraproducts of the special orthogonal groups and the orthogonal groups with the projectivesed rank metric are equal.
\end{cor} 

\begin{lemma}\label{identificationranknorm}
There is the following identification of the rank norm using the standard maximal tori of $\SO(n)$. For $g \in \SO(n)$ the norm $\|g\|_{\rk}$ is twice the number of non-trivial rotations of a representative $hgh^{-1}$ in the maximal torus of $\SO(n)$. In particular, $\|g\|_{\rk}$ is always even. 
\end{lemma}

\begin{proof}
We have 
\begin{align*} hgh^{-1}=
\begin{pmatrix}
 R_1 &  &   \\
  & \ddots &  \\
  &   & R_n  \\
\end{pmatrix}, 
\end{align*}
where all $R_i$ are rotations belonging to $\SO(2)$. An element in $\SO(2)$ has a fixed point if and only if it is the identity. Thus, the largest subspace in $\R^n$ without any fixed point of $g$ has dimension twice the number of non-trivial rotations. Consequently, $\|g\|_{\rk}$ is twice the number of non-trivial rotations $R_i$.   
\end{proof}

Recall, that the action of $\SO(n)$ on the unit sphere $\mathbb{S}^{n-1}$ gives rise to a fibration with fibre $\SO(n-1)$. In this spirit we want to define projection maps $\SO(n) \rightarrow \SO(n-1)$ that rotate vectors back into this fibre and do not increase the norm $\|.\|_{\rk}$.

\begin{lemma} \label{rotationnorm}
Let $g,h \in \SO(n)$ such that $g(e_n) \neq h(e_n)$. Let $R_\alpha$ and $R_\beta$ be elementary rotations such that $ R_\alpha g(e_n)=e_n=  R_\beta h(e_n)$ and such that $R_\alpha$ (and $R_\beta$ respectively) is the identity if $g(e_n)=e_n$ (or $h(e_n)$ respectively). Then $\| R_\alpha g ( R_\beta h)^{-1} \|_{\rk} \leq \| gh^{-1} \|_{\rk}$. 
\end{lemma}

\begin{proof}
By conjugation invariance $\| R_\alpha g( R_\beta h)^{-1} \|_{\rk} = \| g h^{-1} R_\beta^{-1} R_\alpha\|_{\rk}$. By our assumptions the rotations $R_\alpha$ and $R_\beta^{-1}$ either both act non-trivially on $e_n$ or one of them is the identity. In both cases their product is again an elementary rotation since every element of $\SO(3)$ is an elementary rotation. Consequently, $\| R_\beta^{-1} R_\alpha\|_{\rk} \leq 2$. Since 
$ R_\alpha g(e_n)=e_n=  R_\beta h(e_n)$, it holds that $g(e_n) \in \Fix( g h^{-1} R_\beta^{-1} R_\alpha)$, but $R_\beta^{-1} R_\alpha (g(e_n))=h(e_n) \neq g(e_n)$. Let $k \in \N$ be the norm $k=\|gh^{-1}\|_{\rk}$. Then $\dim(\Fix(gh^{-1}))=n-k$ and thus
\begin{align*}
\dim \Big (\Fix( g h^{-1} R_\beta^{-1} R_\alpha ) \Big) & \geq \dim \Big ( \big( \Fix(gh^{-1}) \cap \Fix(R_\beta^{-1} R_\alpha) \big) \oplus \langle g(e_n) \rangle \Big )\\
  = & \dim \Big (  \Fix(gh^{-1}) \cap \Fix(R_\beta^{-1} R_\alpha) \Big )+1
\\  = & \dim(\Fix(gh^{-1}))+\dim(\Fix(R_\beta^{-1} R_\alpha))- \dim(\Fix(gh^{-1}) \\
& + \Fix(R_\beta^{-1} R_\alpha) ) + 1 \\
 \geq &  n-k + n-2 - n +1 \\
 = & n-(k+1) .
\end{align*}
Therefore, $\| R_\alpha g( R_\beta h)^{-1} \|_{\rk} \leq \|gh^{-1}\|_{\rk} +1$. However, by Lemma \ref{identificationranknorm} the rank norms $\| R_\alpha g( R_\beta h)^{-1} \|_{\rk}$ and $\|gh^{-1}\|_{\rk}$ are both even, so $\|gh^{-1}\|_{\rk}+1$ is odd. Consequently, we must have $\| R_\alpha g( R_\beta h)^{-1} \|_{\rk} \leq  \|gh^{-1}\|_{\rk}$. 
\end{proof}

With these preliminary considerations at hand we will define our desired projection maps $p_n \colon \SO(n) \rightarrow \SO(n-1)$ as follows. For every $x \in \mathbb{S}^{n-1}$ we fix an elementary rotation $R_x \in \SO(n)$ which rotates $x$ to $e_n$ and set $R_{e_n}=id$. Note, that this choice is unique when $x \neq \pm e_n$ since $x$ and $e_n$ span a 2 dimensional subspace in that case. Define
\begin{align*}
p_n \colon \SO(n) \rightarrow \SO(n-1) \hspace{2mm} \text{by} \hspace{2mm}
p_n(g)= R_{g(e_n)} g. 
\end{align*}
Note, that $R_{g(e_n)} g$ is indeed an element of $\SO(n-1)$ since it belongs to $SO(n)$ and satisfies $R_{g(e_n)} g(e_n)=e_n$. Thus, it follows by orthogonality that the last entry in all other columns of $R_{g(e_n)} g$ vanishes and it lies in the image of $\SO(n-1)$. 

\begin{theorem}\label{SOncontractible}
The metric ultraproduct $\prod_{n \to \omega} (\SO(n),\|.\|_{\overline{\rk}})$ of the special orthogonal groups with the projectivised rank metric is contractible. 
\end{theorem}

\begin{proof}
We will verify the conditions of Lemma \ref{GeneralContractionLemma} for the cone with respect to the standard scaling sequence $\Cone_\omega(\SO(n), \|.\|_{\rk}) = \prod_{n \to \omega} (SO(n),\|.\|_{\overline{\rk}})$ to obtain a contracting homotopy. 

Clearly, $\|p_n(g)g^{-1} \|_{\rk} \leq 2$ for all $n \in \N$ and $g \in \SO(n)$. Additionally, by Lemma \ref{rotationnorm} all $g,h \in \SO(n)$ satisfy $\|p_n(g)p_n(h)^{-1}\|_{\rk} \leq \|gh^{-1}\|_{\rk}$ if $g(e_n) \neq h(e_n)$. However, if $g(e_n)=h(e_n)$ then $R_{g(e_n)}=R_{h(e_n)}$ and so
\begin{align*}
\| p_n(g) p_n(h)^{-1} \|_{\rk} = \| R_{g(e_n)}gh^{-1} R_{h(e_n)}^{-1} \|_{\rk} = \|gh^{-1} \|_{\rk}.
\end{align*}
Therefore, the result follows from Lemma \ref{GeneralContractionLemma}.
\end{proof}

\begin{remark}
The analogeous statement for unitary groups does not follow since the rank norm on $\SU(n)$ can assume odd values.   
\end{remark}

\section{Appendix: Non-contractible metric ultraproducts and asymptotic cones}

The purpose of this appendix is to give some immediate examples for groups that have non-contractible asymptotic cones stemming directly from the metric ultraproduct of cyclic groups of order $n$. More background and much more involved theory focussing mainly on the case of finitely generated groups can be found in \cite{Burillo}, \cite{Riley} and \cite{Osinagain}.

\begin{lemma}\label{ultracircle}
Let $G_n= \Z/n$ be equipped with the word norm $\|.\|_n$ generated by $1 \in \Z/n$ for all $n$ and let $\|.\|_{\overline{n}}$ be the projectivisation. Then $\prod_{n \to \omega} (\Z/n,\|.\|_{\overline{n}})$ is Lipschitz isomorphic to the unit circle $S^1$ with its standard metric. 
\end{lemma}

\begin{proof}
Recall, that the standard subspace metric on $S^1$ is Lipschitz equivalent to the metric given by the arclength $d_{\arc}$. For each $n$ consider the homomorphism \[
\varphi_n \colon \Z/n \rightarrow S^1 \hspace{5mm} , \hspace{5mm} k \rightarrow e^{2 \pi i \frac{k}{n}} 
\] 
with image equal to the $n$-th roots of unity. Then $a,b \in \Z/n$ with $\|ab^{-1}\|_n=k$ satisfy $d_{\arc}(\varphi_n(a),\varphi_n(b)) = 2 \pi \frac{k}{n}$. Hence, the sequence $(\varphi_n)$ gives rise to a Lipschitz continuous homomorphism $\varphi \colon \Cone_\omega(\Z/n, \|.\|_n) \rightarrow S^1$ setting $\varphi(x)= \lim_\omega \varphi_n (x_n)$ for $x=[x_n]$.

Conversely, define $\theta_n \colon S^1 \rightarrow \Z/n$ by $\theta_n(x)=k$ where $k$ is a fixed choice for each $x$ such such that $d_{\arc}(x,e^{2 \pi i \frac{k}{n}})$ is minimal. These (non-continuous) maps satisfy 
\[
\| \theta_n(x) \theta_n(y)^{-1} \|_n \leq d_{\arc}(x,y)n+2.
\] 
Consequently, setting $\theta(x)=[\theta_n(x)]$ yields a map $\theta \colon S^1 \rightarrow \Cone_\omega(\Z/n, \|.\|_n)$ that is Lipschitz continous. Since for each $n$ the choice of $k$-th root of unity in the definition of $\theta_n(x)$ is distance minimising, this sequence converges again to $x \in S^1$ and so $(\varphi \circ \theta)=id$. 

For $[x_n] \in \Cone_\omega(\Z/n, \|.\|_n)$ let $x^\prime=\lim_\omega \varphi_n (x_n)$ and let $y_n = \theta_n(x^\prime)$ for all $n$. We claim that $[y_n]=[x_n]$. 
Let $\epsilon>0$ be given. Then the following sequence of implications holds where the first line follows from the definition of the ultralimit $x^\prime=\lim_\omega \varphi_n (x_n)$.
\begin{align*}
& \left \{ n \in \mathbb{N} \hspace{2mm} : \hspace{2mm} d_{\arc}(e^{2 \pi i \frac{x_n}{n}}, x^\prime)  < \epsilon \right \} \in \omega \\
\implies & \left \{ n \in \mathbb{N} \hspace{2mm} : \hspace{2mm} \| \theta_n(e^{2 \pi i \frac{x_n}{n}}) \theta_n(x^\prime)^{-1} \|_n  < \epsilon n +2 \right \} \in \omega \\
\implies & \left \{ n \in \mathbb{N} \hspace{2mm} : \hspace{2mm} \frac{ \| x_n   y_n^{-1} \|_n}{n}  < \frac{ \epsilon n +2}{n} \right \} \in \omega \\
\implies & \left \{ n \in \mathbb{N} \hspace{2mm} : \hspace{2mm} \frac{ \| x_n   y_n^{-1} \|_n}{n}  < 2 \epsilon \right \} \in \omega.
\end{align*}
Therefore, $[x_n]=[y_n]$ and so $(\theta \circ \varphi)= id$. Thus, $\varphi$ and $\theta$ are Lipschitz isomorphisms and  $\prod_{n \to \omega} (\Z/n,\|.\|_{\overline{n}})$ is Lipschitz isomorphic to the unit circle $S^1$ with its standard metric.
\end{proof}

\begin{theorem}\label{circleproduct}
Let $G= \bigoplus_n \Z/n $ be equipped with the word norm $\|.\|$ generated by the unit vectors $\{ 1_n \in \Z/n \}_n$. Then $\Cone_\omega(G,\|.\|)$ is Lipschitz isomorphic to $S^1 \times \Cone_\omega(X_n,\|.\|_n)$ for suitable sequence $(X_n, \|.\|_n)$. 
\end{theorem}

\begin{proof}
Let $p_n \colon G \rightarrow \Z/n$ be the projection on the $\Z/n$ factor and let $\iota_n \colon \Z/n \rightarrow G$ be the inclusion. Then for all $n$ the inclusion $\iota_n$ is an isometric embedding, $p_n$ is Lipschitz continuous with constant 1 and $p_n \circ \iota_n= id_{\Z_n}$. On the level of asymptotic cones this induces the maps $p \colon \Cone_\omega(G,\|.\|) \rightarrow \Cone_\omega(\Z/n, \|.\|_n)$ and $\iota \colon \Cone_\omega(\Z/n, \|.\|_n) \rightarrow \Cone_\omega(G,\|.\|)$. Thus, the exact sequence of abelian groups 
\begin{figure}[H]
\centering
\begin{tikzcd}
0 \arrow[-latex]{r} & \ker(p) \arrow[-latex]{r}
 & \Cone_\omega(G,\|.\|) \arrow[-latex]{r}{p}&   \Cone_\omega(\Z/n, \|.\|_n) \arrow[-latex, bend right=40,swap]{l}{\iota} \arrow[-latex]{r} & 0.
\end{tikzcd}
\end{figure}
splits as a product of groups. The maps
\begin{align*}
& \varphi \colon \Cone_\omega(G,\|.\|) \rightarrow \Cone_\omega(\Z/n, \|.\|_n) \times \ker(p), &  & g \rightarrow \left( p(g), g-p(g) \right), \\
& \psi \colon \Cone_\omega(\Z/n, \|.\|_n) \times \ker(p) \rightarrow \Cone_\omega(G,\|.\|),  & & (a,b) \rightarrow \iota(a)+b.
\end{align*}
display the product structure on the level of groups. However, all these involved maps are identities or Lipschitz continuous and so $\varphi$ and $\psi$ are Lipschitz continuous as well. The result follows from Lemma \ref{ultracircle}. 

The kernel $\ker(p)$ is given precisely by equivalence classes having a representative sequence $(x_n)$ such that $x_n \notin \Z/n$ for all $n$. Hence, setting $X_n =  \bigoplus_{k \neq n} \Z/k$ and equipping this space with the natural subspace metric $d_n$ generated by $\{ 1_k \in \Z/k \ : k \neq n \}_k$ it holds that $\ker(p)= \Cone_\omega(X_n, d_n)$. Finally, it follows that $\psi$ is an isometric isomorphism. 
\end{proof}

\begin{cor}
Let $G= \bigoplus_n \Z/n $ be equipped with the word norm generated by $\{ 1_n \in \Z/n \}_n$. There is an injective map $\pi_1(S^1) \hookrightarrow \pi_1(\Cone_\omega(G,\|.\|))$. In particular, $\Cone_\omega(G,\|.\|)$ is not contractible. 
\end{cor}

\begin{proof}
Every map collapsing $S^1$ to a point and including it in another topological space $Y$ is continuous. Therefore, by the universal property of the product there are continuous maps $S^1 \hookrightarrow S^1 \times \Cone_\omega(X_n,\|.\|_n) \rightarrow S^1$ whose composition is the identity. This induces an injective map $\pi_1(S^1) \hookrightarrow \pi_1(\Cone_\omega(G,\|.\|))$ on the level of fundamental groups. 
\end{proof}

\begin{theorem}\label{circleretract}
Let $\Z/n$ be equipped with the word norm $\|.\|_n$ generated by $1 \in \Z/n$ for all $n$ and let the free product $\ast_{k \in \N} \Z/k$ be equipped with the associated $\ell_1$-norm. Then $S^1$ is a topological retract of $\Cone_\omega(\ast_{k \in  \N} \Z/k, \|.\|_{\ell_1})$.  
\end{theorem}

\begin{proof}
By Lemma \ref{inclfreeprod} the inclusions $\iota_n \colon \Z/n \rightarrow \ast_k \Z/k$ are all isometries. The projections $p_n \colon \ast_k \Z/k \rightarrow \Z/n$ are all Lipschitz continuous with constant one. On the level of asymptotic cones this gives rise to the Lipschitz homomorphisms
\begin{align*}
\Cone_\omega(\Z/n, \|.\|_n) \xrightarrow{\iota=[\iota_n]}  \Cone_\omega(\ast_{k} \Z/k, \|.\|_{\ell_1}) \xrightarrow{p=[p_n]} \Cone_\omega(\Z/n, \|.\|_n).  
\end{align*}
By Lemma \ref{ultracircle} identify $\Cone_\omega(\Z/n, \|.\|_n)$ with the unit circle. Since $p_n \circ \iota_n= id_{\Z/n}$ for all $n$ we have $p \circ \iota=id_{S^1}$ and so $\Cone_\omega(\ast_{k \in  \N} \Z/k, \|.\|_{\ell_1})$ retracts onto $S^1$ via $p$.   
\end{proof}

\begin{cor}
Let $\Z/n$ be equipped with the word norm $\|.\|_n$ generated by $1 \in \Z/n$ for all $n$ and let $\ast_{k \in \N} \Z/k$ be equipped with the associated $\ell_1$-norm. Then $\Cone_\omega(\ast_{k \in  \N} \Z/k, \|.\|_{\ell_1})$ has nontrivial fundamental group. In particular, it is not contractible. \qed
\end{cor}

\section*{Acknowledgements}

I like to thank Andreas Thom and Jakob Schneider for their helpful comments during a research visit at the TU Dresden in January 2020. This work was partly funded by the Leverhulme Trust Research Project Grant RPG-2017-159.

\end{document}